\documentclass{amsart}

\usepackage{amssymb}
\usepackage{graphicx}
\usepackage{color}
\usepackage{multicol}

\DeclareMathOperator{\supp}{supp}
\DeclareMathOperator{\SL}{SL}
\DeclareMathOperator{\Lip}{Lip}
\DeclareMathOperator{\Sol}{Sol}
\DeclareMathOperator{\diam}{diam}

\DeclareMathOperator{\mass}{mass}
\DeclareMathOperator{\FV}{FV}
\DeclareMathOperator{\rk}{rk}
\newcommand{\dimAN}{\dim_{\text{AN}}}
\newcommand{\CLip}{C^\text{Lip}}
\newcommand{\Ccell}{C^\text{cell}}

\newcommand{\N}{\ensuremath{\mathbb{N}}}
\newcommand{\R}{\ensuremath{\mathbb{R}}}
\newcommand{\Hyp}{\ensuremath{\mathbb{H}}}
\newcommand{\Q}{\ensuremath{\mathbb{Q}}}
\newcommand{\Z}{\ensuremath{\mathbb{Z}}}

\renewcommand{\vec}{\mathbf}

\DeclareMathOperator{\dlk}{\text{Lk}^\downarrow_\infty}
\DeclareMathOperator{\Lk}{Lk}
\DeclareMathOperator{\downlk}{\text{Lk}^\downarrow}
\newcommand{\dcomb}{\ensuremath{d_{\text{comb}}}}
\newcommand{\fa}{\ensuremath{\mathfrak{a}}}
\newcommand{\fb}{\ensuremath{\mathfrak{b}}}
\newcommand{\fc}{\ensuremath{\mathfrak{c}}}
\newcommand{\fd}{\ensuremath{\mathfrak{d}}}
\newcommand{\pinfty}{\ensuremath{\partial_\infty}}
\newcommand{\xop}{{X_\infty^0(\fa)}}

\newcommand{\V}{\ensuremath{\mathcal{V}}}

\usepackage[pdftex,pdftitle={Lipschitz connectivity and filling invariants in solvable groups and buildings},pdfauthor={Robert Young}]{hyperref}

\newtheorem{thm}{Theorem}[section]
\newtheorem{theorem}[thm]{Theorem}

\newtheorem{lemma}[thm]{Lemma}

\newtheorem{cor}[thm]{Corollary}
\newtheorem{conj}[thm]{Conjecture}
\theoremstyle{remark}

\title[Lipschitz connectivity and filling invariants]{Lipschitz connectivity and filling invariants in solvable groups and buildings}
\author{Robert Young}
\date{\today}

\begin{document}
\bibliographystyle{alpha}
\begin{abstract}
  We give some new methods, based on Lipschitz extension theorems, for
  bounding filling invariants of subsets of nonpositively curved
  spaces.  We apply our methods to find sharp bounds on higher-order
  Dehn functions of $\Sol_{2n+1}$, horospheres in euclidean buildings,
  Hilbert modular groups, and certain $S$-arithmetic groups.
\end{abstract}

\maketitle
\section{Introduction}
Filling invariants of a group or space, such as Dehn functions and
higher-order Dehn functions, are quantitative versions of finiteness
properties.  There are many methods for bounding the Dehn function,
but bounds on the Dehn function are often difficult to generalize to
higher-order Dehn functions.  For example, one can prove that a
non-positively curved space has a Dehn function which is at most
quadratic in a couple of lines: the fact that the distance function is
convex implies that the disc formed by connecting every point on the
curve to a basepoint on the curve has quadratically large area.  On
the other hand, proving that a non-positively curved space has a
$k$th-order Dehn function bounded by $V^{(k+1)/k}$ takes several pages
\cite{GroFRM, WengerShort}.  In this paper, we describe some new
methods for bounding higher-order Dehn functions and apply them to
solvable groups and subsets of nonpositively curved spaces.

One reason that higher-order Dehn functions are harder to bound is
that the geometry of spheres is more complicated than the geometry of
curves.  A closed curve is geometrically very simple.  It has diameter
bounded by its length, it has a natural parameterization by length,
and a closed curve in a space with a geometric group action can be
approximated by a word in the group.  None of these hold for spheres.
A $k$-sphere of volume $V$ may have arbitrarily large diameter, has no
natural parameterization, and, though it can often be approximated by
a cellular or simplicial sphere, that sphere may have arbitrarily many
cells of dimension less than $k$.

One way around this is to consider Lipschitz extension properties.  A
typical Lipschitz extension property is Lipschitz $k$-connectivity; we
say that a space $X$ is \emph{Lipschitz $k$-connected} (with constant
$c$) if there is a $c$ such that for any $0\le d\le k$ and any
$l$-Lipschitz map $f:S^d\to X$, there is a $cl$-Lipschitz extension
$\bar{f}:D^{d+1}\to X$.  The advantage of dealing with Lipschitz
spheres rather than spheres of bounded volume is that techniques for
filling closed curves often generalize to Lipschitz spheres.  For
example, the same construction that shows that a non-positively curved
space has quadratic Dehn function shows that such a space is Lipschitz
$k$-connected for any $k$.  Any map $f:S^d\to X$ can be extended to a
map $\bar{f}:D^{d+1}\to X$ by coning $f$ off to a point along
geodesics, and if $f$ is Lipschitz, so is $\bar{f}$.

In this paper, we describe a way to use Lipschitz connectivity to
prove bounds on higher-order filling functions of subsets of spaces
with finite Assouad-Nagata dimension.  These spaces include euclidean
buildings and homogeneous Hadamard manifolds \cite{LangSch}, and we
will show that a higher-dimensional analogue of the
Lubotzky-Mozes-Raghunathan theorem holds for Lipschitz $n$-connected
subsets of spaces with finite Assouad-Nagata dimension.  Recall that
Lubotzky, Mozes, and Raghunathan proved that
\begin{theorem}\cite{LMR}\label{thm:LMR}
  If $\Gamma$ is an irreducible lattice in a semisimple group $G$ of
  rank $\ge 2$, then the word metric on $\Gamma$ is quasi-isometric to
  the restriction of the metric on $G$ to $\Gamma$.
\end{theorem}
One way to state this theorem is that the inclusion
$\Gamma\hookrightarrow G$ does not induce any distortion of lengths.
That is, there is a $c>0$ such that if $x,y\in \Gamma$ are connected
by a path of length $l$ in $G$, then they are connected by a path of
length $\le cl$ in the Cayley graph of $\Gamma$.  We can think of
this as an efficient 1-dimensional filling of a 0-sphere.  Many authors have
conjectured that when $G$ has higher rank, we can fill
higher-dimensional spheres efficiently; for example, Thurston famously
conjectured that $\SL(n;\Z)$ has quadratic Dehn function for $n\ge 4$
\cite{GerstenSurv}, and Gromov conjectured that the $(k-2)$-th order
Dehn function of a lattice in a symmetric space of rank $k$ should be
bounded by a polynomial \cite{GroAI}.  Bux and Wortman
\cite{BuxWortFiniteness} conjectured that
filling volumes should be undistorted in lattices in higher-rank
semisimple groups.  We will state a version of this conjecture in
terms of homological filling volumes; in a highly-connected space,
these are equivalent to homotopical filling volumes in dimensions
above 2 \cite{GroFRM,White,Groft}.

To state the conjecture, we introduce Lipschitz chains.  A Lipschitz $d$-chain in
$Y$ is a formal sum of Lipschitz maps $\Delta_d\to Y$.  One can define
the boundary of a Lipschitz chain as for singular chains, and this
gives rise to a homology theory.  If $\alpha$ is a Lipschitz $d$-cycle
in $Y$, define
$$\FV^{d+1}_Y(\alpha)=\inf_{\partial \beta=\alpha}\mass \beta.$$
to be the filling volume of $\alpha$ in $Y$.  In particular, if $Y$ is
a geodesic metric space and $\alpha$ is the 0-cycle $x-y$, then
$\FV^1_Y(\alpha)=d(x,y)$.  

If $Z\subset X$, we say
that $Z$ is \emph{undistorted up to dimension $n$} if there is some
$r\ge 0$ and $c$ such that if $\alpha$ is a Lipschitz $d$-cycle in
$Z$ and $d<n$, then
$$\FV^{d+1}_{Z}(\alpha) \le c\FV^{d+1}_X(\alpha)+c.$$ 
(Note that this differs from Bux and Wortman's definition in
\cite{BuxWortFiniteness}; Bux and Wortman's definition deals with
extending spheres in a neighborhood of $Z$ to balls in a larger
neighborhood.)
\begin{conj}[see \cite{BuxWortFiniteness}, Question 1.6]\label{conj:BW}
  If $\Gamma$ is an irreducible lattice in a semisimple group $G$ of
  rank $n$, then there is a nonempty $\Gamma$-invariant subset
  $Z\subset G$ such that $d_{\text{Haus}}(Z,\Gamma)<\infty$ and $Z$ is
  \emph{undistorted up to dimension $n-1$}.
\end{conj}
Here, $d_{\text{Haus}}(Z,\Gamma)$ represents the Hausdorff distance
between the two sets.

Theorem~\ref{thm:LMR} is a special case of this conjecture.  As
Bestvina, Eskin, and Wortman note in \cite{BestEskWort}, Conj.\
\ref{conj:BW} would imply that the $k$th-order Dehn function of
$\Gamma$ is bounded by $V^{(k+1)/k}$.  In recent years, a significant
amount of progess has been made toward these conjectures.  Dru\c{t}u
proved that a lattice of $\Q$-rank 1 in a symmetric space of $\R$-rank
$\ge 3$ has a Dehn function bounded by $n^{2+\epsilon}$ for any
$\epsilon>0$ \cite{DrutuFilling}, Leuzinger and Pittet proved that,
conversely, any irreducible lattice in a symmetric space of rank 2
which is not cocompact has an exponentially large Dehn function
\cite{LeuzPitRk2}, and the author proved Thurston's conjecture in the
case that $n\ge 5$ \cite{YoungQuad}.

In this paper, we make a step toward proving Conj.\ \ref{conj:BW} by
showing that, under some conditions on $G$ and $\Gamma$,
undistortedness follows from a Lipschitz extension property.  We say
that $Z$ is \emph{Lipschitz $n$-connected} if there is a $c$ such that
for any $0\le d\le k$ and any $l$-Lipschitz map $f:S^d\to Z$, there is
a $cl$-Lipschitz extension $\bar{f}:D^{d+1}\to Z$.  If $Z\subset Y$,
we say $Z$ is \emph{Lipschitz $n$-connected in $Y$} if, under the
above conditions, there is a $cl$-Lipschitz extension
$\bar{f}:D^{d+1}\to Y$.

\begin{theorem}\label{thm:LipToLMR}
  Suppose that $Z\subset X$ is a nonempty closed subset with metric
  given by the restriction of the metric of $X$.  Suppose that $X$ is
  a geodesic metric space such that the Assouad-Nagata dimension
  $\dimAN(X)$ of $X$ is finite.  Suppose that one of the following is true:
  \begin{itemize}
  \item $Z$ is Lipschitz $n$-connected.
  \item $X$ is Lipschitz $n$-connected, and if $X_p,p\in P$ are the connected components of
    $X\smallsetminus Z$, then the sets $H_p=\partial X_p$ are
    Lipschitz $n$-connected with uniformly bounded
    implicit constant.
  \end{itemize}
  Then $Z$ is undistorted up to dimension $n+1$.  
\end{theorem}
In the applications in this paper, $X$ will be a CAT(0) space (either
a symmetric space or a building), and $Z$ will either be a horosphere
of $X$ or the complement of a set of disjoint horoballs.

When $X$ is CAT(0), a theorem of Gromov \cite{GroFRM, WengerShort}
implies that the $k$th-order Dehn function of $X$ grows at most as
fast as $V^{(k+1)/k}$ (i.e., if $\alpha$ is a Lipschitz $k$-cycle in
$X$, there is a Lipschitz $(k+1)$-chain $\beta$ in $X$ such that
$\partial \beta=\alpha$ and
$$\mass \beta\lesssim (\mass \alpha)^{(k+1)/k}+c.$$  
Therefore,
\begin{cor}\label{cor:LipToLMR}
  If $X$ is CAT(0) and the hypotheses above hold, the $k$th-order Dehn
  function of $Z$ grows at most as fast as $V^{(k+1)/k}$ for $k\le n$.
\end{cor}
This bound is often sharp; for instance, if there is a rank-$(k+1)$
flat of $X$ contained in $Z$, then the $k$th-order Dehn function of
$Z$ grows at least as fast as $V^{(k+1)/k}$.

We will apply Theorem~\ref{thm:LipToLMR} to find fillings in a family of solvable groups and in the Hilbert modular groups:
\begin{theorem}\label{thm:Sol}
  The group $\Sol_{2n-1}=\R^{n-1}\ltimes \R^n$ is Lipschitz
  $n-1$-connected, is undistorted in $(\Hyp_2)^n$ up to dimension
  $n$, and its $k$th-order Dehn function is asymptotic to
  $V^{(k+1)/k}$ for $k<n$.
\end{theorem}
This is a higher-dimensional version of a theorem of Gromov
\cite[5.A$_9$]{GroAI} which states that $\Sol_{2n-1}$ has quadratic
Dehn function when $n>1$.  These bounds are sharp; there are
$n$-spheres in $\Sol_{2n-1}$ with volume $V$ but filling volume
exponential in $V$, so the $n$th order Dehn function of $\Sol_{2n-1}$
is exponential \cite{GroAI}.  The same bounds apply to Hilbert modular
groups:
\begin{thm}\label{thm:HilbertModular}
  If $\Gamma\subset \SL_2(\R)^n$ is a Hilbert modular group, then the
  $k$th-order Dehn function of $\Gamma$ is asymptotic to
  $V^{(k+1)/k}$ for $k<n$.
\end{thm}

We will also apply the methods of Theorem~\ref{thm:LipToLMR} to
horospheres in euclidean buildings and to the $S$-arithmetic groups
considered in \cite{BuxWortConnectivity}.  

Let $X$ be a thick euclidean building and $E\subset X$ be an
apartment.  Then the vertices of $E$ form a lattice, and if
$r:[0,\infty)\to E$ is a geodesic ray, we say that $r$ has
\emph{rational slope} if it is parallel to a line segment connecting
two vertices of $E$.  This condition is independent of the choice of
$E$, so if $r:[0,\infty)\to X$ is a geodesic ray, we say it has
rational slope if it has rational slope considered as a ray in some
apartment $E$.  The boundary at infinity of $X$ consists of
equivalence classes of geodesic rays, so if $\tau\in X_\infty$ is a
point in the boundary at infinity of $X$, we say it has rational slope
if one of the rays asymptotic to $\tau$ has rational slope.  In
particular, if the isometry group of $X$ acts cocompactly on a
horosphere centered at $\tau$, then $\tau$ has rational slope.

\begin{thm}\label{thm:mainBuildings}
  Let $X$ be a thick euclidean building and let $\tau\in X_\infty$ be
  a point in its boundary at infinity which has rational slope and is
  not contained in a factor of rank less than $n$ (in particular, $X$
  has rank at least $n$).  Let $Z$ be a horosphere in $X$ centered at
  $\tau$.  Then $Z$ is Lipschitz $(n-2)$-connected, undistorted in $X$
  up to dimension $n-1$, and its $k$th-order Dehn function grows at
  most as fast as $V^{(k+1)/k}$ for $k\le n-2$.
\end{thm}
$Z$ is not $(n-1)$-connected, so the bound on $k$ is sharp.  Indeed,
for every $r>0$, there is a map $\alpha:S^{n-1}\to Z$ such that $\alpha$
is not null-homotopic in the $r$-neighborhood of $Z$ (see
Lemma~\ref{lem:nonFiniteness}). 

Note that if $\tau$ does not have rational slope, then $Z$ may be
$(n-2)$-connected and locally Lipschitz $(n-2)$-connected but not
Lipschitz $(n-2)$-connected.  Cells of $X$ may intersect $Z$ in
arbitrarily small sets, and this can lead to arbitrarily small spheres
which have small fillings in $X$ but filling volume $\sim 1$ in $Z$.

Theorem~\ref{thm:mainBuildings} is similar to Theorem~7.7 of
\cite{BuxWortConnectivity}, and gives a higher-order version of
Theorem~1.1 of \cite{DrutuFilling} for buildings and products of
buildings.  (Though note that Theorem~1.1 of \cite{DrutuFilling}
applies to $\R$-buildings as well as discrete buildings.)  

The same methods lead to bounds on the higher-order Dehn functions of $S$-arithmetic groups of $K$-rank 1.
\begin{thm}\label{thm:SArith}
  Let $K$ be a global function field, $\mathbf{G}$ be a
  noncommutative, absolutely almost simple $K$-group of $K$-rank 1,
  let $S$ be a finite set of pairwise inequivalent valuations on $K$,
  and let $X$ be the associated euclidean building.  Then the 
  $k$th-order Dehn function of the $S$-arithmetic group
  $G(\mathcal{O}_S)$ grows at most as fast as $V^{(k+1)/k}$ for $k\le
  \dim X-2$.
\end{thm}
This improves results of Bux and Wortman, who showed that
$G(\mathcal{O}_S)$ is of type $F_{\dim X-1}$ but not of type $F_{\dim
  X}$ \cite{BuxWortConnectivity, BuxWortFiniteness}.  Bux and Wortman
showed that horospheres in $X$ are $(\dim X-2)$-connected;
Theorem~\ref{thm:SArith} gives a quantitative proof of this fact.

Some possible other applications of Theorem~\ref{thm:LipToLMR} include
the study of higher-order fillings in, for instance, metabelian
groups, as in \cite{DeCorTess}, lattices of $\Q$-rank 1 in semisimple
groups, as in \cite{DrutuFilling}, and $S$-arithmetic lattices when
$|S|$ is large, as in \cite{BestEskWort}.

\noindent \emph{Notational conventions}: If $f$ and $g$ are
expressions, we will write $f\lesssim g$ if there is some constant $c$
such that $f\le c g$.  We write $f\sim g$ if there is some constant
$c$ such that $c^{-1}\le f\le c g$.  When we wish to emphasize that
$c$ depends on $x$ and $y$, we write $f \lesssim_{x,y} g$ or
$f\sim_{x,y} g$.  We give $S^k$ the round metric, scaled so that
$\diam S^k=1$, and we define the \emph{standard $k$-simplex} to be the
equilateral euclidean $k$-simplex, scaled to have diameter 1.

\noindent \emph{Acknowledgements}: This work was supported by a
Discovery Grant from the Natural Sciences and Engineering Research
Council of Canada and by the Connaught Fund, University of Toronto.
The author would like to thank MSRI and the organizers of the 2011
Quantitative Geometry program for their hospitality, and would like to
thank Cornelia Dru\c{t}u, Enrico Leuzinger, Romain Tessera, and Kevin
Wortman for helpful discussions and suggestions.

\section{Building fillings from simplices}

The proof of Theorem~\ref{thm:LipToLMR} is based on the proof of a
theorem of Lang and Schlichenmaier.  Lang and Schlichenmaier proved:
\begin{theorem}\label{thm:LangSch}
  Suppose that $Z\subset X$ is a nonempty closed set and that $\dimAN{X}\le m<\infty$.  If $Y$
  is Lipschitz $(m-1)$-connected, then there is a $C$ such that any
  Lipschitz map $f:Z\to Y$ extends to a map $\bar{f}:X\to Y$ with
  $\Lip(\bar{f})\le C\Lip(f)$.
\end{theorem}
Here, $\dimAN(X)$ is the Assouad-Nagata dimension of $X$.  The
\emph{Assouad-Nagata dimension} of $X$ is the smallest integer such
that there is a $c$ such that for all $s>0$, there is a covering
$\mathcal{B}_s$ of $X$ by sets of diameter at most $cs$ (a
\emph{$cs$-bounded covering}) such that any set with diameter $\le s$
intersects at most $n+1$ sets in the cover (i.e., $\mathcal{B}_s$ has
\emph{$s$-multiplicity} at most $n+1$).

One consequence of Theorem~\ref{thm:LangSch} is that if $Z$ is
Lipschitz $n$-connected for $n=\dimAN(X)$, then the identity map $Z\to
Z$ can be extended to a Lipschitz map $\bar{f}:X\to Z$ and $Z$ is a
Lipschitz retract of $X$.  Consequently, if $\alpha$ is a
$(k-1)$-cycle in $Z$ and $\beta$ is a chain in $X$ with boundary
$\alpha$, then $\bar{f}_\sharp(\beta)$ is a chain in $Z$ with boundary
$\alpha$, so
$$\FV^{k}_Z(\alpha)\le C^k \FV_X^k(\alpha),$$
and $Z$ is undistorted in $X$ up to dimension $n$.
Theorem~\ref{thm:LipToLMR} claims that the same is true under the
weaker condition that $X$ has finite Assouad-Nagata dimension.  

Before we sketch the proof of Theorem~\ref{thm:LipToLMR}, we need the
notion of a quasi-conformal complex.  We define a \emph{riemannian
  simplicial complex} to be a simplicial complex with a metric which
gives each simplex the structure of a riemannian manifold with
corners.  We say that such a complex is quasi-conformal (or that the
complex is a \emph{QC complex}) if there is a $c$ such that the
riemannian metric on each simplex is $c$-bilipschitz equivalent to a
scaling of the standard simplex.

QC complexes are a compromise between the rigidity of simplicial
complexes and the freedom of riemannian simplicial complexes.  A key
feature of simplicial complexes is that curves and cycles can be
approximated by simplicial curves and cycles.  This is not true in
riemannian simplicial complexes, but it holds in QC complexes.

Specifically, a version of the Federer-Fleming deformation theorem
holds in QC complexes.  Recall that the Federer-Fleming theorem for
simplicial complexes states that any Lipschitz cycle $a$ in a
simplicial complex can be approximated by a simplicial cycle $P(a)$
whose mass is comparable to the mass of $a$.  We will use the
following variation of the Federer-Fleming theorem:
\begin{theorem}\label{thm:FedFlemSimp}
  Let $\Sigma$ be a finite-dimensional scaled simplicial complex, that
  is, a simplicial complex where each simplex is given the metric of
  the standard simplex of diameter $s$.  There is a constant $c$
  depending on $\dim \Sigma$ such that if $a\in \CLip_k(\Sigma)$ is a
  Lipschitz $k$-cycle, then there are $P(a)\in \Ccell_k(X)$ and
  $Q(a)\in \CLip_{k+1}(X)$ such that
  \begin{enumerate}
  \item $\partial a=\partial P(a)$
  \item $\partial Q(a) = a - P(a)$
  \item $\mass P(a)\le c \cdot\mass(a)$
  \item $\mass Q(a)\le c s \cdot\mass(a)$
  \end{enumerate}
\end{theorem}
A proof of this theorem when $s=1$ can be found in \cite{ECHLPT}.  A
simple scaling argument proves the general case.  Note that, while the
bound on $\mass Q(a)$ depends on the size of the simplices, the bound
on $\mass P(a)$ does not.

Because the bound
on $\mass P(a)$ is independent of the size of the simplices in the
complex, the following version of
Theorem~\ref{thm:FedFlemSimp} holds for a QC complex:
\begin{theorem}\label{thm:FedFlemConf}
  Let $\Sigma$ be a QC complex.  There is a constant $c$
  depending on $\dim \Sigma$ such that if $a\in \CLip_k(\Sigma)$ is a
  Lipschitz $k$-cycle, then there is a $P(a)\in \Ccell_k(X)$ such that
  $\partial a=\partial P(a)$ and $\mass P(a)\le c \cdot\mass(a)$.
\end{theorem}

Now we will sketch a proof of Theorem~\ref{thm:LipToLMR}.  Note that
this sketch is incorrect due to some technical issues; we will fix
these issues in the actual proof.  In the proof of Theorem~1.5 of
\cite{LangSch}, Lang and Schlichenmaier show that, if
$\dimAN(X)<\infty$, there are $a>0$, $0<b<1$ and a cover
$\mathcal{B}=(B_i)_{i\in I_0}$ of $X\setminus Z$ by subsets of
$X\setminus Z$ such that:
\begin{enumerate}
\item $\diam B_i\le a d(B_i,Z)$ for every $i\in I_0$
\item every set $D\subset X\setminus Z$ with $\diam D\le b
  d(D,Z)$ meets at most $\dimAN(X)+1$ members of $(B_i)_{i\in I_0}$.
\end{enumerate}
They then define functions $\sigma_i:X\setminus Z\to \R$,
$$\sigma_i(x)=\max\{0,\delta d(B_i,Z)-d(x,B_i)\},$$
where $\delta=b/(2(b+1))$, and show that these have the
property that for any $x$, there are no more than $\dimAN(X)+1$ values
of $i$ for which $\sigma_i(x)>0$.  Using these $\sigma_i$, they
construct a Lipschitz map $g:X\setminus Z\to \Sigma_0$, where
$\Sigma_0$ is the nerve of the supports of the $\sigma_i$.  One can
give $\Sigma_0$ the structure of a QC complex so that if $\Delta$ is a
simplex of $\Sigma_0$ with a vertex corresponding to $\sigma_i$, then
$\diam \Delta \sim \diam \supp \sigma_i$.  Since the diameter of
$\supp \sigma_i$ is proportional to $d(\sigma_i,Z)$, this means that
the parts of $\Sigma_0$ which are close to $Z$ are given a fine
triangulation and the parts of $\Sigma_0$ which are far from $Z$ are
given a coarse triangulation.

Since $Z$ is Lipschitz $n$-connected, one can construct a Lipschitz
map $h:\Sigma_0^{(n+1)}\to Z$, where $\Sigma_0^{(n+1)}$ is the
$(n+1)$-skeleton of $\Sigma_0$.  Then, if $\alpha$ is an $n$-cycle in
$Z$, it has a filling $\beta$ in $X$.  We can use the Federer-Fleming
theorem to approximate $g_\sharp(\beta)$ by some simplicial
$(n+1)$-chain $P(\beta)$ which lies in $\Sigma_0^{(n+1)}$.  The
pushforward of $P(\beta)$ under $h$ will then be a filling of
$\alpha$.

The problem with this argument is twofold.  First, since $g$ is only
defined on $X\setminus Z$, we can't define $g_\sharp(\beta)$ without
extending $g$ to $Z$.  We could define an appropriate metric on the disjoint union
$\Sigma_0\amalg Z$ and a map $X\to \Sigma_0 \amalg Z$, but this is no
longer a simplicial complex.  Second, since the cells of $\Sigma$ get
arbitrarily small close to $Z$, $P(\beta)$ may be an infinite sum of
cells of $\Sigma$.  

We know of two ways to fix this issue.  First, one can make sense of
infinite sums of cells of $\Sigma$ by introducing Lipschitz currents
\cite{AmbrosioKirchheim}.  The set of Lipschitz currents is a
completion of the set of Lipschitz chains, and the $P(\beta)$ defined
above is a current in $\Sigma_0 \amalg Z$.  Its pushforward is then a
filling of $\alpha$.  Second, we can change the construction of
$\Sigma_0$ to avoid the problem.  We take this approach in the rest of
this section.

All the constants and all the implicit constants in $\lesssim$ and
$\sim$ in this section will depend on $X, Z$, and $n$.

First, we construct a QC complex $\Sigma$ which approximates $X$.
This complex will have geometry similar to $\Sigma_0$ on $X\setminus
Z$ and it will have $\epsilon$-small simplices on $Z$.  For $t>0$, let
$N_t(Z)\subset X$ be the $t$-neighborhood of $Z$.
\begin{lemma}
  There are $a, b,\gamma>0$ such that if $\epsilon>0$ and
  $\delta=b/(2(b+1))$, there is a covering $\mathcal{D}$ of $X$ by
  sets $D_k$, $k\in K$ and functions $r:K\to \R$, $\tau_k:X\to \R$
  $$r(k)=\max \{\delta d(D_k,Z), \epsilon \}$$
  $$\tau_k(x)=\max\{0, r(k)-d(x,D_k)\}$$
  such that for any $k\in K$,
  \begin{enumerate}
  \item $\diam D_k \lesssim r(k),$
  \item $d(D_k,Z) \lesssim r(k),$
  \item if $\rho=\epsilon\delta(1+a)$ and $d(D_k,Z)\ge \rho$, then
    $\supp \tau_k$ is contained in a connected component of
    $X\smallsetminus Z$,
  \item the cover of $X$ by the sets $\supp \tau_k$ has multiplicity
    at most $2\dimAN(X)+2$, and 
  \item if $\supp \tau_k\cap \supp \tau_{k'}\ne \emptyset$, then 
    $$\gamma^{-1}r(k')\le  r(k)\le \gamma r(k').$$
  \end{enumerate}
\end{lemma}
\begin{proof}
  Let $a>0$, $0<b<1$, and $\mathcal{B}=(B_i)_{i\in I_0}$ be as in the
  Lang-Schlichenmaier construction above.  Let 
  We may assume that each $B_i$ is contained in a connected component
  of $X\smallsetminus Z$.  Let $\rho=\epsilon\delta(1+a)$, and let
  $I\subset I_0$ be the set
  $$I:=\{i\in I_0 \mid B_i \not\subset N_\rho(Z)\}.$$
  Then $\bigcup_{i\in I} B_i\supset X\setminus N_\rho(Z)$.  Since
  $\dimAN(X)\le \infty$, we can let $\mathcal{C}=\{C_j\}_{j\in J}$ be
  a $2 c_0\epsilon$-bounded covering of $N_{\rho}(Z)$ with 
  $2\epsilon$-multiplicity at most $\dimAN(X)+1$, where $c_0$ is the constant in the
  definition of $\dimAN(X)$.  Let $\mathcal{D}=\mathcal{C}\cup
  \{B_i\}_{i\in I}$ and let $K=I\amalg J$.

  Conditions (1) and (2) are easy to check.  For (3), note that if
  $d(D_k,Z)\ge \rho$, then $k\in I$, so $D_i=B_i$ lies in a single
  connected component of $X\setminus Z$, and $\supp \tau_i$ lies in
  the same component.  For (4), note that if $i\in I$, then
  $\tau_i=\sigma_i$, so the cover $\{\supp \tau_i\}_{i\in I}$ has
  multiplicity at most $\dimAN(X)+1$.  Likewise, if $x\in \supp
  \tau_j$ for some $j\in J$, then $C_j\cap B(x,\epsilon)\ne
  \emptyset$, where $ B(x,\epsilon)$ is the closed ball of radius
  $\epsilon$ around $x$.  Since $\mathcal{C}$ has bounded
  $2\epsilon$-multiplicity, this can be true for only $\dimAN(X)+1$
  values of $j$.

  To check (5), suppose that $\supp \tau_k\cap \supp
  \tau_{k'}\ne \emptyset$.  If $r(k')=\epsilon$, then $r(k')\le r(k)$.
  Otherwise, $r(k')=\delta d(D_{k'},Z)$.  But $d(D_k,D_{k'})\lesssim
  r(k)$ and $\diam D_{k}\lesssim r(k)$, so $d(D_{k'},Z)\lesssim r(k)$,
  and $r(k')\lesssim r(k)$.  By symmetry, $r(k)\sim r(k')$.
\end{proof}

Let $\Sigma$ be the nerve of the cover $\{\supp \tau_k\}_{k\in K}$,
with vertex set $\{v_k\}_{k\in K}$ and let $s:\Sigma\to \R$ is the
function such that $s(v_k)=r(k)$ and $s$ is linear on each simplex of
$\Sigma$.  Define a riemannian metric $x_c$ on each simplex of
$\Sigma$ by letting $dx_c^2=s^2\;dx^2$.  If $S=\langle v_{k_1},\dots
v_{k_n}\rangle$ is a simplex of $\Sigma$, then $s$ varies between
$\gamma^{-1}r(k_1)$ and $\gamma r(k_1)$ on $S$, so this metric makes
$\Sigma$ a QC complex.   

\begin{lemma}
  There is a Lipschitz map $g:X\to \Sigma$ with Lipschitz constant
  $c_1$ independent of $\epsilon$.  Furthermore, if $x\in \supp
  \tau_k$ for some $k\in K$, then $g(x)$ is in the star of $v_k$.
\end{lemma}
\begin{proof}
  Consider the infinite simplex
  $$\Delta_K:=\{p:K\to [0,1]\mid \|p\|_1=1\}$$
  with vertex set $K$, so that $\Sigma$ is a subcomplex of
  $\Delta_K$.  Let
  $$g(x)(k)=\frac{\tau_k(x)}{\bar{\tau}(x)},$$
  where $\bar{\tau}(x)=\sum_{k\in K} \tau_k(x)$.  The image of $g$
  then lies in $\Sigma$, and we can consider $g$ as a function $X\to
  \Sigma$.

  It remains to show that $g$ is Lipschitz with respect to the QC
  metric on $\Sigma$.  Since $X$ is geodesic, it suffices to show that
  if $x,y\in X$ and $d(x,y)<\delta^2 \epsilon<\epsilon,$
  then $d(g(x),g(y))\lesssim d(x,y)$.  Let $S$ and $T$ be the minimal
  simplices of $\Sigma$ which contain $g(x)$ and $g(y)$ respectively.
  First, we claim that $S$ and $T$ share some vertex $v_{m}$.

  Let $\rho=\epsilon\delta(1+a)$ as above.
  If $d(x,Z)<\rho$, then there is some $j\in J$ such that $x\in C_j$
  and $\tau_j(x)=\epsilon$.  Since $\tau_j$ is 1-Lipschitz,
  $\tau_j(y)>0$, so we can let $m=j$.  Otherwise, if $d(x,Z)\ge \rho$,
  then there is some $i\in I$ such that $x\in B_i$.  We have
  $$d(x,Z) \le \diam(B_i)+d(B_i,Z) \le (a+1)d(B_i,Z),$$
  so $\tau_i(x)=\delta d(B_i,Z)\ge \delta^2 \epsilon$, and $\tau_i(y)>0$ as
  desired.  We let $m=i$.

  Since $S$ and $T$ share $v_{m}$, the value of $s$ on $S\cup T$ is
  at most $\gamma r(m)$, and 
  \begin{align*}
    d(g(x),g(y))& \le \gamma r(m) \sum_{k\in
      (S\cup T)^{(0)}}\left|\frac{\tau_k(x)}{\bar\tau(x)}-\frac{\tau_k(y)}{\bar\tau(y)}\right|\\
    & \le \gamma r(m) \sum_{k\in
      (S\cup T)^{(0)}}\left|\frac{\tau_k(x)}{\bar\tau(x)}-\frac{\tau_k(y)}{\bar\tau(x)}\right|
    + \left|\frac{\tau_k(y)}{\bar\tau(x)}-\frac{\tau_k(y)}{\bar\tau(y)}\right| \\
    & \le \gamma r(m) \sum_{k\in
      (S\cup
      T)^{(0)}}\frac{1}{\bar{\tau}(x)}\left(|\tau_k(x)-\tau_k(y)|+\frac{\tau_k(y)}{\bar{\tau}(y)}|\bar{\tau}(x)-\bar{\tau}(y)|\right)\\
    & \le \gamma (2 \dim \Sigma+1)(2\dim \Sigma + 2) \frac{r(m)}{\bar{\tau}(x)}d(x,y)
  \end{align*}
  Furthermore, if $x\in D_{m'}$, then 
  $$\bar{\tau}(x)\ge r(m')\ge \gamma^{-1}r(m),$$
  so $g$ has Lipschitz constant at most 
  $$c_1=\gamma^2 (2 \dim \Sigma+1)(2\dim \Sigma + 2).$$
\end{proof}

Next, we construct a map $h:\Sigma^{(n+1)}\to Z$ on the
$(n+1)$-skeleton of $\Sigma$.  If $\Delta$ is a simplex of $\Sigma$,
denote its vertex set by $\V(\Delta)$.
\begin{lemma}
  For any $\epsilon'>0$, there is a Lipschitz map
  $h^{(0)}:\Sigma^{(0)}\to Z$ with Lipschitz constant independent of
  $\epsilon$ which satisfies:
  \begin{enumerate}
  \item $d(h^{(0)}(v_j),C_j)\lesssim \epsilon$ for every $j\in J$,
  \item if $X_p,p\in P$ are the connected components of
    $X\smallsetminus Z$ and
    $$H_p(\epsilon')=\{x\in X\mid d(x,X_p)\le \epsilon'\} \cap Z,$$
    then for any simplex $\Delta\subset \Sigma$, we either have $\diam
    h^{(0)}(\V(\Delta))\lesssim \epsilon$ (if $\Delta$ has a vertex of
    the form $v_j$ for some $j\in J$) or $h^{(0)}(\V(\Delta))\subset
    H_p(\epsilon')$ for some $p\in P$ (otherwise).
  \end{enumerate}
\end{lemma}
\begin{proof}
  For each vertex $v_k\in \Sigma$, choose a point $z_k\in Z$ such that
  $d(z_k,D_k)< d(Z,D_k)+\epsilon_H/2$, and let $h^{(0)}(v_k)=z_k$.  If
  $k\in J$, then $d(Z,D_k)\lesssim \epsilon$, so $d(z_k,D_k)\lesssim
  \epsilon$ and property (1) holds.
  We claim that this map is Lipschitz.  Suppose that $v,w$ are
  vertices of $\Sigma$.  Then there is a path $\gamma:[0,1]\to \Sigma$
  between them of length $\ell(\gamma)\le 2 d(v,w)$, and the
  Federer-Fleming theorem implies that this can be approximated by a
  path $\gamma':[0,1]\to \Sigma^{(1)}$ in the 1-skeleton of $\Sigma$
  with $\ell(\gamma')\lesssim \ell(\gamma)$.  So, to check that
  $h^{(0)}$ is Lipschitz, it suffices to show that if $v_k$ and
  $v_{k'}$ are connected by an edge $e$, then $d(z_k,z_{k'})\lesssim
  \ell(e)$.

  We may assume that $r(k)\ge r(k')$, so $\ell(e)\ge \gamma^{-1}r(k)$.
  Then we can bound $d(z_k,z_{k'})$ by
  $$d(z_k,z_{k'})\le d(z_k,D_k)+\diam(D_k)+d(D_k,D_{k'})+\diam(D_{k'})+d(D_{k'},z_{k'})$$
  Each term on the right is $\lesssim r(k)$.  For each term except
  $d(D_k,D_{k'})$, this follows from the remarks after the definition
  of $S$. To bound $d(D_k,D_{k'})$, note that since there is an edge
  from $v_k$ to $v_{k'}$, there is a $w\in \supp \tau_k\cap \supp
  \tau_{k'}$.  Then $d(w,D_k)<r(k)$ and $d(w,D_{k'}) < r(k)$, so
  $d(D_k,D_{k'})\le 2r(k)$.  Therefore, $h^{(0)}$ is Lipschitz.

  It remains to check property (2).  Let $\Delta=\langle
  v_{k_0},\dots,v_{k_n}\rangle$ be a simplex of $\Sigma$ and suppose
  that $k_i\in J$ for some $i$.  Then $r(k_i)\lesssim \epsilon$, so
  $\diam \Delta\lesssim \epsilon$, and therefore, $\diam
  h^{(0)}(\V(\Delta))\lesssim \epsilon$.

  Otherwise, $k_i\in I$ for all $i$.  Then there is some $p\in P$ such
  that $\supp \tau_{k_i}\subset X_p$ for all $i$, and
  $h^{(0)}(\V(\Delta))\subset H_p(\epsilon')$.
\end{proof}

If $\epsilon>0$ and $n$ are such that whenever $k\le n$ and
$\tau:S^k\to Z$ is a map with $\Lip \tau \le \epsilon$, there is an
extension $\bar\tau:D^{k+1}\to Z$ with $\Lip \bar\tau \lesssim \Lip
\tau$, we say that $Y$ is \emph{$\epsilon$-locally Lipschitz $n$-connected}.
\begin{lemma}
  If $X$ and $Z$ satisfy the conditions of Theorem~\ref{thm:LipToLMR}
  and $\epsilon$ is sufficiently small, then there is a Lipschitz
  extension $h:\Sigma^{(n+1)}\to Z$ with Lipschitz constant independent of
  $\epsilon$ such that $d(h(g(z)),z)\lesssim \epsilon$ for every $z\in
  Z$.
\end{lemma}
\begin{proof}
  In this proof, it will be convenient to let $S^k$ be the boundary of
  the standard $(k+1)$-simplex and $D^k$ be the standard $k$-simplex.
  If $t>0$, we let $tS^k$ and $tD^k$ be scalings of $S^k$ and $D^k$.
  If a space $Y$ is Lipschitz $n$-connected, there is a $c$ such
  that if $k\le n$, any Lipschitz map $\tau:S^k\to Y$ can be
  extended to a Lipschitz map $\bar{\tau}:D^{k+1}\to Y$ with $\Lip
  \bar \tau\le c\Lip \tau$.  By scaling, any Lipschitz map $\tau:t S^k\to Y$ can be
  extended to a Lipschitz map $\bar{\tau}:t D^{k+1}\to Y$ with $\Lip
  \bar \tau\le c\Lip \tau$

  If $Z$ is Lipschitz $n$-connected, then we can use Lipschitz
  $n$-connectivity to extend $h^{(0)}$.  That is, if we have already
  defined $h$ on $\Sigma^{(k)}$ and $\Delta\subset \Sigma$ is a
  $(k+1)$-simplex, then the Riemannian metric on $\Delta$ is
  bilipschitz equivalent to $s(x)D^{k+1}$ for any $x\in \Delta$.
  Since $h|_{\partial \Delta}$ is a Lipschitz map of a $k$-sphere, we
  can extend $h$ over $\Delta$, and the extension satisfies $\Lip
  h\lesssim \Lip h^{(0)}$.

  If $Z$ is not Lipschitz $n$-connected, we need a more careful
  approach.  By hypothesis, $X$ is Lipschitz $n$-connected; let $c>0$
  be the constant in the definition of Lipschitz $n$-connectivity.

  Let $\epsilon'=\epsilon_H/c$ and let $k\le n$.  If $\tau:S^k\to Z$ is
  a map with $\Lip \tau \le \epsilon'$, we claim that $\tau$ can be
  extended to a Lipschitz map on $D^{k+1}$.  If $\tau(S^k)\subset
  H_p(\epsilon_H)$ for some $p\in P$, then we can extend $\tau$ to
  $D^{k+1}$ using the Lipschitz $n$-connectivity of $H_p(\epsilon_H)$.
  Otherwise, there is some $x\in S^k$ such that $d(\tau(x),X\setminus
  Z)>\epsilon_H$.  Since $\diam \bar{\tau}_0(D^{k+1})\le \epsilon_H$,
  the image of $\bar{\tau}_0$ is contained in $Z$.  Therefore, $Z$ is
  $\epsilon'$-locally Lipschitz $n$-connected.


  If $\Delta\subset \Sigma$ is a simplex, we say that it is
  \emph{coarse} if all its vertices are of the form $v_i$ for $i\in
  I$.  We say that it is \emph{fine} if it has a vertex of the form
  $v_j$ for some $j\in J$; all fine simplices have diameter $\lesssim
  \epsilon$ and all coarse ones have diameter $\gtrsim \epsilon$.  By
  the previous lemma, we can choose $h^{(0)}$ so that for every coarse
  simplex $\Delta$, there is some $p\in P$ such that
  $h^{(0)}(\V(\Delta))\subset H_p(\epsilon_H)$.  If $\Sigma_c\subset
  \Sigma$ is the subcomplex consisting of coarse simplices, we can
  extend $h^{(0)}$ to a map
  $h_c:\Sigma^{(0}\cup \Sigma^{(n+1)}_c\to Z$ by induction; if $h_c|_{\partial\Delta}$
  is defined, then $h_c(\partial \Delta)\subset H_p(\epsilon_H)$ for
  some $p\in P$.    We extend $h_c$ over $\Delta$ using the Lipschitz
  $n$-connectivity of $H_p(\epsilon_H)$.  The Lipschitz constant of
  $h_c$ is bounded independently of $\epsilon$.  

  Again by the previous lemma, we may choose $\epsilon$ sufficiently
  small that any fine simplex has diameter $\ll \epsilon'/\Lip h_c$.
  We can then extend $h_c$ over the fine simplices of $\Sigma$ using
  the local Lipschitz connectivity of $Z$ to get the desired map $h$.

  In either case, if $z\in Z$, then $z\in \supp \tau_k$ only if $k\in
  J$.  In particular, $g(z)$ is contained in a fine simplex of
  diameter $\lesssim \epsilon$ and $d(z,z_i)\lesssim \epsilon$, so
  $$d(h(g(z)),z)\le d(h(g(z)),h(v_i))+d(z_i,z)\lesssim \epsilon$$
  as desired.
\end{proof}

Therefore, $h\circ g$ has small displacement.  To complete the proof
of Theorem~\ref{thm:LipToLMR}, we will need a lemma concerning
such maps:
\begin{lemma}\label{lem:smallDisplacement}
  Suppose that $m\le n$, that $\alpha$ is a Lipschitz $m$-cycle in
  $Z$, that $Z$ is $\epsilon_0$-locally Lipschitz $n$-connected, and
  that $C>0$.  For any $\epsilon>0$, there is a $\delta>0$ such that
  if $f:Z\to Z$ is a $C$-Lipschitz map with displacement $\le \delta$
  (i.e., $d(f(z),z)\le \delta$ for all $z\in Z$), then
  $$\FV_Z^{m+1}(f_\sharp(\alpha)-\alpha)\le \epsilon.$$
\end{lemma}
\begin{proof}
  Since $Z$ is locally Lipschitz $n$-connected, if $M$ is a simplicial
  $(m+1)$-complex, $N$ is a subcomplex, and $f:N\to Z$ is a map with
  sufficiently small Lipschitz constant, then there is an extension
  $\bar{f}:M\to Z$ with Lipschitz constant $\sim \Lip(f)$.  Write
  $\alpha$ as a sum $\alpha=\sum_{i=1}^k\alpha_i$ of Lipschitz maps
  $\alpha_i:\Delta^m\to Z$.  Let $L$ be the maximum Lipschitz constant
  of the $\alpha_i$'s.  In the following calculations, all our
  implicit constants will depend on $k$, $n$, $Z$, $C$, and $L$.
  We claim that
  $$\FV_Z^{m+1}(f_\sharp(\alpha)-\alpha)\lesssim \delta.$$
  
  First, we can subdivide $\Delta^m$ into $\lesssim \delta^{-m}$
  simplices each with diameter $\le \delta/L$.  We can use this
  subdivision to replace $\alpha$ with a sum
  $\alpha'=\sum_{i=1}^{k'}\alpha'_i$ where $k'\lesssim \delta^{-m}$
  and each $\alpha'_i:\Delta^m\to Z$ has Lipschitz constant at most
  $\delta$.

  There is a simplicial $m$-complex $A$ with at most $k'$
  top-dimensional faces, a simplicial cycle $\omega$ on $A$, and a map
  $g:A\to Z$ with $\Lip(g)\le \delta$ such that the restriction of $g$
  to each top-dimensional face of $A$ is one of the $\alpha'_i$'s and
  $g_\sharp(\omega)=\alpha'$.  Define $r_0:A \times \{0,1\}\to Z$ by
  letting $r_0|_{A\times 0}=g$ and $r_0|_{A\times 1}=f\circ g$.  Then
  $\Lip(r_0)\lesssim \delta$, and if $\delta$ is sufficiently small, we
  can extend it to a Lipschitz map $r:A\times [0,1]\to
  Z$ with $\Lip r\sim \Lip r_0$.  This is a homotopy from $g$ to $f\circ g$, so the push-forward
  of $\omega\times[0,1]$ is a filling of $f_\sharp(\alpha)-\alpha$ with
  mass
  $$\mass r_\sharp(\omega\times[0,1])\lesssim k' \delta^{m+1}\lesssim \delta$$
  as desired.
\end{proof}

\begin{proof}[Proof of Theorem~\ref{thm:LipToLMR}]
  Suppose that $\alpha$ is a $(m-1)$-cycle in $Z$ and $\beta$ is a
  $m$-chain filling it.  Let $\Sigma_J$ be the subcomplex of $\Sigma$
  spanned by the vertices $v_j, j\in J$.  Then $g(Z)\subset \Sigma_J$,
  and $g_\sharp(\alpha)$ is a cycle in $\Sigma_J$ with mass $\le
  \Lip(g)^{m-1} \mass \alpha$.  Each simplex of $\Sigma_J$ has
  diameter $\sim \epsilon$, so by Thm.~\ref{thm:FedFlemSimp}, there is
  a $c_3>0$ depending only on $X$, a simplicial cycle
  $P_\alpha:=P_{\Sigma_J}(g_\sharp(\alpha))$ approximating
  $g_\sharp(\alpha)$, and a chain
  $Q_\alpha:=Q_{\Sigma_J}(g_\sharp(\alpha))$ such that $\mass
  Q_\alpha\le c_3 \epsilon \mass(\alpha)$ and $\partial
  Q_\alpha=P_\alpha-g_\sharp(\alpha)$.

  Then $g_\sharp(\beta)+Q_\alpha$ is a $m$-chain in $\Sigma$ with boundary
  $P_\alpha$ and mass
  $$\mass(g_\sharp(\beta)+Q_\alpha)\le \Lip(g)^{m}\mass\beta+c_3 \epsilon \mass(\alpha).$$
  Thm.~\ref{thm:FedFlemConf} lets us approximate this by a chain
  $$P_\beta:=P_{\Sigma}(g_\sharp(\beta)+Q_\alpha)$$
  with boundary $P_\alpha$.  

  By Lemma~\ref{lem:smallDisplacement}, if $\epsilon_0>0$, then
  for $\epsilon$ sufficiently small, there is a Lipschitz
  $m+1$-chain $R$ in $Z$ such that
  $$\partial R=\alpha-(h\circ g)_\sharp(\alpha)$$
  and $\mass R\le \epsilon_0$.  Let 
  $$B=R-h_\sharp(Q_\alpha)-h_\sharp(P_\beta)$$
  Then $\partial B=\alpha$
  and
  $$\mass B \lesssim \mass\beta+\epsilon \mass(\alpha)+\epsilon_0,$$
  so 
  $$\FV_Z^k(\alpha)\lesssim \mass\beta$$
  as desired.
\end{proof}

The rest of this paper is dedicated to applying this theorem to
horospheres and lattices in symmetric spaces and buildings.

\section{Fillings in $\Sol_{2n-1}$}\label{sec:sol}

Theorem~\ref{thm:LipToLMR} is useful because it reduces a
difficult-to-prove statement about the undistortedness of an inclusion
to an easier-to-prove Lipschitz extension property.  For example, in
this section, we will prove:
\begin{theorem}\label{thm:mainSol}
  The solvable Lie group $\Sol_{2n-1}=\R^{n-1}\ltimes \R^n$ is Lipschitz
  $(n-2)$-connected.
\end{theorem}
Theorem~\ref{thm:Sol} follows as a direct application of
Theorem~\ref{thm:LipToLMR}.

We start by defining $\Sol_{2n-1}$, $n\ge 2$.  This group is a
solvable Lie group which can be written as a semidirect product of
$\R^n$ and $\R^{n-1}$, where $\R^{n-1}$ acts on $\R^n$ as the group of
$n\times n$ diagonal matrices with positive coefficients and
determinant 1.  When $n=2$, this is the three-dimensional solvable
group corresponding to solvegeometry.

All the constants and implicit constants in this section will depend
on $n$.

One feature of this group is that it can be realized as a horosphere
in a product of hyperbolic planes.  Let $\Hyp_2$ be the hyperbolic
plane and let $\beta:\Hyp_2\to \R$ be a Busemann function for
$\Hyp_2$.  We can define Busemann functions $\beta_1,\dots,\beta_n$ in
the product $\Hyp_2^n$ by letting $\beta_i(x_1,\dots,x_n)=\beta(x_i)$.
Then $b=n^{-1/2}\sum\beta_i$ is a Busemann function for $\Hyp_2^n$,
and $\Sol_{2n-1}$ acts freely, isometrically, and transitively on the
resulting horosphere $b^{-1}(0)$.  The metric induced on $\Sol_{2n-1}$
by inclusion in $\Hyp_2^n$ is bilipschitz equivalent to a
left-invariant Riemannian metric on $\Sol_{2n-1}$.

This group also appears as a subgroup of a Hilbert modular group.  If
$\Gamma\subset \SL_2(\R)^n$ is a Hilbert modular group and
$X=(\Hyp_2)^n$, then there is a collection $\mathcal{H}$ of disjoint
open horoballs in $X$ such that the boundary of each horosphere is
bilipschitz equivalent to $\Sol_{2n-1}$ and $\Gamma$ acts cocompactly
on $X\smallsetminus \mathcal{H}$ \cite{PittetHMG}.  Consequently,
Theorem~\ref{thm:HilbertModular} is also a corollary of
Theorem~\ref{thm:mainSol}.

To prove Theorem~\ref{thm:mainSol}, we will use the following
condition, which is equivalent to Lipschitz connectivity (see
\cite{GroCC}):
\begin{lemma} \label{lem:LipToInfSimp} Let $Z$ be a metric space, let
  $\Delta_Z$ be the infinite-dimensional simplex with vertex set $Z$,
  and let $\Delta_Z^{(k)}$ be its $k$-skeleton.  Let
  $\langle z_0,\dots, z_k\rangle$ denote the $k$-simplex with vertices
  $z_0,\dots,z_k$.  Then $Z$ is Lipschitz $n$-connected if and only if
  there exists a map $\Omega:\Delta_Z^{(n+1)}\to Z$ such that
  \begin{enumerate}
  \item For all $z\in Z$, $\Omega(\langle z\rangle)=z$.
  \item There is a $c$ such that for any $d\le n+1$ and any simplex
    $\delta=\langle z_0,\dots,z_d\rangle$, we have
    $$\Lip \Omega|_{\delta}\le c \diam \{z_0,\dots,z_d\}.$$
  \end{enumerate}
\end{lemma}
\begin{proof}
  One direction is clear; if $Z$ is Lipschitz $n$-connected, then one
  can construct $\Omega$ by letting $\Omega(\langle z \rangle)=z$ for all
  $z\in Z$, then using the Lipschitz connectivity of $Z$ to extend
  $\Omega$ over each skeleton inductively.  

  The other direction is an application of the Whitney decomposition.
  We view $D^{d+1}$ as a subset of $\R^{d+1}$; by the Whitney covering
  lemma, the interior of $D^{d+1}$ can be decomposed into a union of
  countably many dyadic cubes such that for each cube $C$, one has
  $\diam C\sim_d d(C,S^d)$.  We can decompose each cube into boundedly
  many simplices to get a triangulation $\tau$ of the interior of
  $D^{d+1}$ where each simplex is bilipschitz equivalent to a scaling
  of the standard simplex.

  We construct a map $h:D^{d+1}\to Z$ using this triangulation.  For
  each vertex $v$ in $\tau$, let $h(v)$ be a point in $S^d$ such that
  $d(v,h(v))$ is minimized.  One can check that $h$ is a Lipschitz map
  from $\tau^{(0)}\to S^d$, so $g_0=\alpha\circ h:\tau^{(0)}\to Z$ is
  a Lipschitz map with $\Lip(g_0)\sim_{d,\Omega} \Lip(\alpha)$.  We
  can extend $g_0$ to a map $g:\tau\to Z$ by sending the simplex
  $\langle v_0,\dots,v_k\rangle $ to the simplex $\Omega(\langle
  g_0(v_0),\dots,g_0(v_k)\rangle)$, and this is also Lipschitz with
  $\Lip(g)\sim_{d,\Omega} \Lip(\alpha)$.

  Finally, we extend $g$ to a map $\beta:D^{d+1}\to Z$ by defining
  $g(v)=\alpha(v)$ when $v\in S^{d}$.  Since the diameter of the
  simplices of $\tau$ goes to zero as one approaches the boundary,
  this extension is continuous and therefore Lipschitz, as desired.
\end{proof}

It therefore suffices to prove the following:
\begin{lemma} \label{lem:infSimplex} Let $\Delta=\Delta_{\Sol_{2n-1}}$ be the
  infinite-dimensional simplex with vertex set $\Sol_{2n-1}$.  There
  is a map $\Omega:\Delta^{(n-1)}\to \Sol_{2n-1}$ which satisfies the
  properties in Lemma~\ref{lem:LipToInfSimp}.  Therefore, $\Sol_{2n-1}$ is Lipschitz $(n-2)$-connected.
\end{lemma}
Our construction is based on techniques from \cite{BestEskWort}; we will
construct $\Omega$ using nonpositively curved subsets of $\Sol_{2n-1}$
called $k$-slices.

Recall that we defined $\Sol_{2n-1}$ as a horosphere in $(\Hyp_2)^n$.
Let $\beta:\Hyp_2\to \R$ be the Busemann function used to define
$\Sol_{2n-1}$ and let $*$ be the corresponding point at infinity.  If
$\gamma$ is a geodesic in $\Hyp_2$ which has one endpoint at $*$, we call
$\gamma$ a \emph{vertical geodesic}.  For $i=1,\dots,n$, let
$s_i\subset \Hyp_2$ be either a vertical geodesic or all of $\Hyp_2$.  If
$k$ of the $s_i$'s are equal to $\Hyp_2$, we call the intersection
$s_1\times \dots\times s_n\cap \Sol_{2n-1}$ a \emph{$k$-slice}.

Suppose that $k<n$ and that $S$ is a $k$-slice; without loss of
generality, we may assume that
$$S=\Hyp_2 \times \dots\times \Hyp_2\times \gamma_1 \times
\dots\gamma_{n-k}\cap \Sol_{2n-1}.$$ 
Then the projection to the first
$n-1$ factors (i.e., all but the last factor) is a homeomorphism from $S$ to $(\Hyp_2)^k \times
\R^{n-k-1}$.  In fact, this map is bilipschitz, so $S$ is bilipschitz
equivalent to a Hadamard manifold.  

If $k<n$, then any $k$-slice is Lipschitz $d$-connected for any $d$:
\begin{lemma}\label{lem:negCurvExtend}
  If $X$ is a Hadamard manifold, it is Lipschitz $n$-connected for any $n$.
\end{lemma}
\begin{proof}
  Let $\alpha:S^n\to X$, and let $v\in S^n$.  Let $(x,r)\in S^n\times
  [0,1]$ be polar coordinates on $D^{n+1}$.  We can construct a map
  $\bar{\alpha}:D^{n+1}\to X$ by letting $\bar{\alpha}(x,r)$ be the
  geodesic from $\alpha(v)$ to $\alpha(x)$.  Because the distance
  function on $X$ is convex, this is a Lipschitz map with Lipschitz
  constant $\lesssim \Lip(\alpha)$.
\end{proof}

If $\tau$ is a polyhedral complex and $f:\tau\to \Sol_{2n-1}$, we say
that $f$ is a \emph{slice map} if the image of every cell $\delta$ of
$\tau$ is contained in a $(\dim \delta)$-slice.  

Our main tool in the proof of Lemma~\ref{lem:infSimplex} is the
following:
\begin{lemma}\label{lemma:sliceExtensions}  
  Let $k<n$.  Suppose that $\sigma$ is a polyhedral complex which is
  bilipschitz equivalent to $S^{k-1}$.  Then there is a $c>0$ and a
  polyhedral complex $\tau$ bilipschitz equivalent to $D^k$ which has
  boundary $\sigma$.  Furthermore, if $f:\sigma\to
  \Sol_{2n-1}$ is a Lipschitz slice map, there is an extension $g:\tau\to
  \Sol_{2n-1}$ which is a slice map with $\Lip(g)\le c\Lip(f)$.
\end{lemma}

The basic idea of the lemma is to first construct a family of
projections along horospheres whose images lie in $(n-1)$-slices, then
construct homotopies between $f$ and its projections.  Gluing these
homotopies together will give a map $\sigma\times[0,n]\to
\Sol_{2n-1}$, and adding a final contraction will extend the map to
all of $\tau$.

Let $\beta:\Hyp_2\to \R$ be the Busemann function used to define
$\Sol_{2n-1}$.  If $\gamma$ is a vertical geodesic in $\Hyp_2$ and
$x\in \Hyp_2$, let $p(x)$ be the unique point on $\gamma$ such that
$\beta(x)=\beta(p(x))$.  This defines a map $p_\gamma:\Hyp_2\to \gamma$.  It
is straightforward to check that $p$ is distance-decreasing and that
$d(x,p(x))\le 2 d(x,\gamma)$.

Suppose that $\vec{x}=(x_1,\dots,x_n)\in (\Hyp_2)^n$.  For $i=1,\dots,
n$, let $\gamma_i$ be a vertical geodesic containing $x_i$, and let
$\beta:\Hyp_2\to \R$ be the Busemann function used to define
$\Sol_{2n-1}$.  For each $i$, let $p_i:\Sol_{2n-1}^n\to \Sol_{2n-1}$ be
the map 
$$p_i(y_1,\dots,y_n)=(y_1,\dots, y_{i-1},p_{\gamma_i}(y_i),y_{i+1},\dots, y_n).$$
Let $S$ be the $0$-slice
$$S=\gamma_1\times\dots\times \gamma_n\cap \Sol_{2n-1}$$
and let $S_i$ be the $(n-1)$-slice
$$S_i=\Hyp_2\times\dots\times \gamma_i\times \dots \times \Hyp_2\cap
\Sol_{2n-1},$$ 
where $\gamma_i$ occurs in the $i$th factor.
It is easy to check the following properties:
\begin{itemize}
\item $p_i$ is distance-decreasing
\item $d(y,p_i(y))\le 2d(y,\vec{x})$ for all $y\in \Sol_{2n-1}$
\item $p_i$ preserves $S$ pointwise
\item If $S'$ is a $d$-slice, then $p_i(S')$ lies in a $d$-slice and
  $S'$ and $p_i(S')$ both lie in the same $(d+1)$-slice.  In
  particular, $y$ and $p_i(y)$ lie in a $1$-slice for every $y\in
  \Sol_{2n-1}$.
\end{itemize}
Then:
\begin{lemma}\label{lem:solSegments}
  For any $i$, if $\sigma$ is a polyhedral complex with $\dim
  \sigma<n$, $f:\sigma\to \Sol_{2n-1}$ is a Lipschitz slice map, and
  $s\in \sigma$ satisfies $f(s)=\vec{x}$, then there is a homotopy
  $g:\sigma\times[0,1]\to \Sol_{2n-1}$ from $f$ to $p_i\circ f$ which
  is a Lipschitz slice map with $\Lip(g)\lesssim \Lip(f)$.
\end{lemma}
\begin{proof}
  We construct $g$ one skeleton at a time.  For any cell $\delta$, the
  image $g(\delta\times [0,1])$ will be contained in the minimal slice
  that contains $f(\delta)$ and $p_i(f(\delta))$.  Since $f(\delta)$
  and $p_i(f(\delta))$ lie in a common $(\dim \delta+1)$-slice, this
  ensures that $g$ is a slice map.

  The map $g$ is already defined on the vertices of $\sigma\times
  [0,1]$ and we claim that it is Lipschitz on the $0$-skeleton.  If
  $l=\Lip(f)$, then $f$ and $p_i\circ f$ are $l$-Lipschitz, and if $v$
  is a vertex of $\sigma$, then
  $$d(f(v),p_i(f(v)))\le 2 d(f(v),\vec{x}) \le 2 l\diam \sigma,$$
  so $g$ is Lipschitz on the vertex set with Lipschitz constant
  $\lesssim l$.

  Now suppose that we have defined $g$ on the $(d-1)$-cells of
  $\sigma\times [0,1]$ and that for any $(d-2)$-cell $\delta$, the
  image $g(\delta\times [0,1])$ is contained in the minimal slice
  that contains $f(\delta)$ and $p_i(f(\delta))$.  Consider a
  cell of the form $\delta\times [0,1]$ for some $(d-1)$-cell $\delta$
  in $\sigma$.  Since $f$ is a slice map, $f(\delta)$ lies
  in some $(d-1)$-slice, so $f(\delta)\cup p_i(f(\delta))$ lies in
  some $d$-slice, and this $d$-slice also contains $g(\partial
  \delta\times [0,1])$ by the inductive hypothesis.  Let $S'$ be the
  minimal slice that contains
  $$g(\partial(\delta\times[0,1]))=f(\delta)\cup
  p_i(f(\delta))\cup g(\partial \delta\times [0,1]).$$ By
  Lemma~\ref{lem:negCurvExtend}, we can extend $g$ over $\delta\times
  [0,1]$ so that it sends $\delta\times [0,1]$ to $S'$.  The extension
  is Lipschitz and the Lipschitz constant is $\lesssim \Lip f$.
\end{proof}

Now we can prove Lemma~\ref{lemma:sliceExtensions}.
\begin{proof}[Proof of Lemma~\ref{lemma:sliceExtensions}]
  Let $\tau$ be the complex $\sigma\times[0,n]\cup C\sigma/\sim$,
  where $[0,n]$ is subdivided into $n$ unit-length edges, $C\sigma$
  is the cone over $\sigma$ and $\sim$ is the relation gluing the base
  of $C\sigma$ to $\sigma\times\{n\}$.  This is bilipschitz equivalent
  to $D^k$.

  Choose a basepoint $v_*\in \sigma$ and suppose that
  $f(v_*)=\vec{x}=(x_1,\dots,x_n)\in (\Hyp_2)^n$.  For $i=0,\dots, n$,
  let $f_{i}=p_{i}\circ\dots\circ p_1\circ f$.  By
  Lemma~\ref{lem:solSegments}, for $i=1,\dots, n$, there is a homotopy
  $g_i:\sigma\times[i-1,i]\to \Sol_{2n-1}$ from $f_{i-1}$ to $f_i$
  which is a Lipschitz slice map with $\Lip(g_i)\lesssim
  \Lip(f)$.  
  Concatenating the $g_i$'s gives a map $\sigma\times [0,n]\to
  \Sol_{2n-1}$ which is a Lipschitz homotopy from $f$ to $f_{n}$.  To
  define $g$, it suffices to extend this map over $C\sigma$, but since
  the image of $f_n$ lies in $S$, we can use
  Lemma~\ref{lem:negCurvExtend} to construct such an extension.  Since
  this extension lies in a 0-slice, it is a slice map, so $g$ is a
  slice map and $\Lip(g)\lesssim \Lip(f)$.
\end{proof}

Lemma~\ref{lem:infSimplex} follows easily:
\begin{proof}[Proof of Lemma~\ref{lem:infSimplex}]
  Let $\Delta_d$ be the standard $d$-simplex.  We define a sequence
  $\tau_i, i=0,\dots,n-1$ of polyhedral complexes homeomorphic to
  $\Delta_i$ and a sequence $\sigma_i, i=0,\dots,n-1$ of polyhedral
  complexes homeomorphic to $\partial \Delta_{i+1}$ inductively.  Let
  $\tau_0$ be a single point.  For each $i\ge 0$, let $\sigma_i$ be
  the complex obtained by replacing each $i$-face of $\partial
  \Delta_{i+1}$ by a copy of $\tau_i$.  Let $\tau_{i+1}$ be the
  complex obtained by applying Lemma~\ref{lemma:sliceExtensions} to
  $\sigma_i$.  This is PL-homeomorphic to $\Delta_i$ and has boundary
  $\sigma_i$.
  
  Let $\Delta'$ be the complex obtained by subdividing each
  $d$-simplex of $\Delta^{(n-1)}$ into a copy of $\tau_d$ and let
  $i:\Delta^{(n-1)}\to \Delta'$ be a bilipschitz equivalence taking
  each simplex to the corresponding copy of $\tau_d$.  We can
  construct a slice map $\Omega':\Delta'\to Z$ by defining
  $\Omega'(\langle x \rangle)=x$ for all $x\in \Sol_{2n-1}$ and using
  Lemma~\ref{lemma:sliceExtensions} inductively to extend $\Omega'$
  over each of the $\tau_d$'s.

  That is, if $\delta=\langle x_0,\dots,x_{d+1} \rangle\subset \Delta$ is
  a $(d+1)$-cell, $\Omega'$ is defined on $i(\partial\delta)$, and   
  $$\Omega|_{i(\partial\delta)}:\sigma_d\to \Sol_{2n-1}$$
  is a slice map with Lipschitz constant $\lesssim \diam
  \{x_0,\dots,x_{d+1}\}$, we may extend it to a slice map on
  $i(\delta)$ using Lemma~\ref{lemma:sliceExtensions}.  The resulting
  map $\Omega|_\delta$ has
  $$\Lip(\Omega|_\delta) \lesssim \diam \{x_0,\dots,x_{d+1}\}$$
  as desired.
\end{proof}

Thus, by Lemma~\ref{lem:LipToInfSimp}, $\Sol_{2n-1}$ is Lipschitz
$(n-2)$-connected, and by Theorem~\ref{thm:LipToLMR}, it is
undistorted up to dimension $n$ inside $\Hyp_2^n$.  Consequently, if
$k<n$ and if $\alpha$ is a Lipschitz $k$-cycle in $\Sol_{2n-1}$,
then 
$$\FV_{\Sol_{2n-1}}^{k+1}(\alpha)\lesssim
\FV_{\Hyp_2^n}^{k+1}(\alpha)\lesssim (\mass \alpha)^{(k+1)/k},$$
as desired.

\section{Fillings in horospheres of euclidean buildings}

In this section, we prove Theorem~\ref{thm:mainBuildings}.  

We claim:
\begin{thm}\label{thm:horosphereLipConn}
  Let $X$ be a thick euclidean building and let $X_\infty$ be the
  Bruhat-Tits building of $X$.  If $X$ is reducible, then $X_\infty$
  is a join of buildings; let $\tau$ be a point in $X_\infty$ which has rational slope and is
  not contained in a join factor of rank less than $n$.  Let $Z$ be a
  horosphere in $X$ centered at $\tau$ and let $p:X_\infty\to
  M(X_\infty)$ be the projection of $X_\infty$ to its model chamber.
  Then $Z$ is Lipschitz $(n-2)$-connected with implicit constant
  depending only on $X$ and $p(\tau)$.
\end{thm}
By Theorem~\ref{thm:LipToLMR} and Corollary~\ref{cor:LipToLMR}, this
implies Theorem~\ref{thm:mainBuildings}.

Furthermore, if $K$ is a global function field, $\mathbf{G}$ is a
noncommutative, absolutely almost simple $K$-group of $K$-rank 1, and
$S$ is a finite set of pairwise inequivalent valuations on $K$, then
$\Gamma=\mathbf{G}(\mathcal{O}_S)$ is an $S$-arithmetic group.  If $X$
is the associated euclidean building and $n$ is its rank, then by
Theorem~3.7 of \cite{BuxWortConnectivity}, there is a collection
$\mathcal{H}$ of pairwise disjoint open horoballs in $X$ such that
$X\smallsetminus \mathcal{H}$ is $\mathbf{G}(\mathcal{O}_S)$-invariant
and cocompact.  By Theorem~\ref{thm:horosphereLipConn}, the boundary
of each of these horoballs is Lipschitz $(n-2)$-connected with a
uniform implicit constant, so Theorem~\ref{thm:LipToLMR} implies
Theorem~\ref{thm:SArith}.

As in \cite[Rem.\ 4.2]{DrutuFilling}, it suffices to consider the case
that $X$ is a thick euclidean building of rank $n$ and that $\tau$ is
not parallel to any factor of $X$.  If $X=X_1\times X_2$, then
$X_\infty=(X_1)_\infty\ast (X_2)_\infty$.  If $\tau\in
(X_1)_\infty$, then $Z=Z_1\times
X_2$, where $Z_1\subset X_1$ is a horosphere of $X_1$ centered at
$\tau$.  If $\alpha:S^k\to Z$ is $c$-Lipschitz, we can replace it with
its projection to $Z_1$ by using an homotopy with Lipschitz constant
$\lesssim c$, so if $Z_1$ is Lipschitz $(n-2)$-connected, so is $Z$.

Therefore, in this section, we will let $X$ be a thick euclidean
building of rank $n$ equipped with its complete apartment system, and
let $X_\infty$ be its Bruhat-Tits building.  We fix a direction at
infinity $\tau\in X_\infty$ which is not contained in any factor of
$X_\infty$, and let $h$ be a Busemann function centered at $\tau$,
with corresponding horosphere $Z=h^{-1}(0)$.  We orient $h$ so that
$h(x)$ increases as $x$ approaches $\tau$; we use this orientation so
that we can treat $h$ as a Morse function on $X$ more easily.

All the constants in this section and its subsections will depend on $X$ and $Z$.

The proof that $Z$ is Lipschitz $(n-2)$-connected is based on
Lemma~\ref{lem:LipToInfSimp}.  Let $\Delta_Z$ be the
infinite-dimensional simplex with vertex set $Z$, and let
$\Delta_Z^{(k)}$ be its $k$-skeleton.  We will show:
\begin{lemma}\label{lem:CoarseInfSimp}
  There exists a map $\Omega:\Delta_Z^{(n-1)}\to Z$ such that
  \begin{enumerate}
  \item For all $z\in Z$, $\Omega(\langle z\rangle)=z$.
  \item For any $d\le n+1$ and any simplex
    $\delta=\langle z_0,\dots,z_d\rangle$, we have
    $$\Lip \Omega|_{\delta}\lesssim  \diam \{z_0,\dots,z_d\}+1.$$
  \end{enumerate}
\end{lemma}
The only difference between the map in
Lemma~\ref{lem:CoarseInfSimp} and the map in
Lemma~\ref{lem:LipToInfSimp} is the bound on $\Lip \Omega|_{\delta}$.
In Lemma~\ref{lem:LipToInfSimp}, $\Lip \Omega|_{\delta}$ is bounded by
a multiple of $\diam \{z_0,\dots,z_d\}$; in
Lemma~\ref{lem:CoarseInfSimp}, it is bounded by a multiple of
$\diam \{z_0,\dots,z_d\}$ and an additive constant.

As a corollary, we have:
\begin{lemma}\label{lem:NbhdRetraction}
  For any $t>0$, there is a Lipschitz map $r_t:h^{-1}((\infty,t])\cap
  X^{(n-1)} \to Z$ which restricts to the identity map on $Z$.
\end{lemma}
\begin{proof}
  Define $r_t$ on $h^{-1}((\infty,0])$ as the closest-point
  projection.  Since horoballs are convex, this is a
  distance-decreasing map.

  To define $r_t$ on $h^{-1}((0,t])\cap X^{(n-1)}$, we view $X$ as a
  polyhedral complex, i.e., a complex whose faces consist of convex
  polyhedra in $\R^n$, glued along faces by isometries.  Then $h$ is
  linear on each face of $X$, so if $P$ is a face of $X$, then the
  intersections $h^{-1}([0,t])\cap P$, $Z\cap P$, and $h^{-1}(t)\cap
  P$ are convex polyhedra.  Since $\tau$ has rational slope, the set
  $h(X^{(0)})$ of possible values of $h$ on the vertices of $X$ is
  discrete, so only finitely many isometry classes of polyhedra occur
  this way, and we can give $Z_t=h^{-1}([0,t])\cap X^{(n-1)}$ the
  structure of a polyhedral complex with only finitely many isometry
  classes of cells.  We subdivide each cell to make $Z_t$ into a
  simplicial complex.  We define $r_t$ on the vertices of $Z_t$ so
  that $d(r_t(v),v)$ is minimized.  If $\Delta$ is a simplex of $Z_t$
  with vertices $v_0,\dots,v_k$, we define
  $$r_t|_\Delta=\Omega|_{\rangle r_t(v_0),\dots,r_t(v_k)}\rangle.$$
  This gives a Lipschitz map with Lipschitz constant depending on the
  size of the smallest simplex in $Z_t$.
\end{proof}
The proof of this lemma is the only place that we use the assumption
that $\tau$ has rational slope.

Given these lemmas, we prove Theorem~\ref{thm:horosphereLipConn} as
follows:
\begin{proof}[{Proof of Theorem~\ref{thm:horosphereLipConn}}]
  Suppose that $\alpha:S^k\to Z$ is a Lipschitz map.  If
  $\Lip(\alpha)\le 1$, we can extend $\alpha$ to a map $\beta:D^{k+1}\to X$
  by coning $\alpha$ to a point along geodesics in $X$.  Since $X$ is
  CAT(0), $\beta$ is Lipschitz with Lipschitz constant $\sim
  \Lip(\alpha)$.  Furthermore, the image of $\beta$ lies in
  $h^{-1}([-1,1])$, so $r_1\circ \beta:D^{k+1}\to Z$ is an extension
  of $\alpha$ with $\Lip(r_1\circ \beta)\sim \Lip(\alpha)$.

  If $\Lip(\alpha)> 1$, let $L\in \N$ be the smallest integer such
  that $L\ge \Lip(\alpha)$, let $D^{k+1}(L)$ be the cube
  $[0,L]^{k+1}\subset \R^{k+1}$, and let $S^{k}(L)=\partial D^k(L)$.
  We view $\alpha$ as a map $S^k(L)\to Z$ with Lipschitz constant
  $\sim 1$ and try to construct an extension to $D^{k+1}(L)$ with
  Lipschitz constant $\sim 1$.

  As in the proof of Lemma~\ref{lem:LipToInfSimp}, the Whitney
  covering lemma implies that $D^{k+1}(L)$ can be decomposed into a
  union of countably many dyadic cubes such that for each cube $C$,
  one has $\diam C\sim d(C,S^k(L))$.  Since these cubes are dyadic,
  each cube of side length less than one is contained in a cube of
  side length 1.  Let $\mathcal{C}$ be the cover of $D^{k+1}(L)$
  obtained by combining cubes of side length less than 1 into cubes of
  side length 1.  Then for each cube $C$ in $\mathcal{C}$, we have
  $\diam C\sim d(C,S^k(L))+1$, and each cube which touches $S^k(L)$
  has side length 1.  We call the cubes that touch $S^k(L)$ the
  \emph{boundary cubes} and we call the rest \emph{interior cubes}.
  We can decompose each cube into boundedly many simplices to get a
  triangulation $\tau$ of $D^{d+1}$ where each simplex is bilipschitz
  equivalent to a scaling of the standard simplex.  Let $\tau_{i}$ be
  the subcomplex of $\tau$ contained in the interior cubes.

  We construct a map $h:S^k(L)\cup \tau_i\to Z$ using this
  triangulation.  If $x\in S^k(L)$, we define $h(x)=f(x)$.  For each
  vertex $v$ in $\tau_i$, let $h(v)$ be a point in $S^d$ such that
  $d(v,h(v))$ is minimized, and if $\Delta=\langle
  v_0,\dots,v_k\rangle$ is a simplex of $\tau_i$, define
  $$h|_\Delta=\Omega|_{\langle  h(v_0),\dots,h(v_k)\rangle}.$$
  Since $\diam \Delta\gtrsim 1$, this is Lipschitz with $\Lip(h)\sim
  1$.

  Since $X$ is CAT(0) and thus Lipschitz $n$-connected, we can extend
  $h$ over the boundary cubes inductively; if $C$ is a face of a
  boundary cube and $h$ is already defined on $\partial C$, we extend
  $h$ over $C$ by coning $h|_{\partial C}$ to a point along geodesics.
  This produces an extension $h:D^{k+1}(L)\to X$ with Lipschitz
  constant $\Lip(h)\sim 1$.

  Finally, since the boundary cubes are all contained in a
  neighborhood of $S^{k}(L)$, their image is contained in a
  neighborhood of $Z$, so if $c$ is large enough, then $r_c\circ
  h:D^k(L)\to Z$ is an extension of $\alpha$ with Lipschitz constant
  $\sim 1$.
\end{proof}

In the rest of this section, we will prove
Lemma~\ref{lem:CoarseInfSimp}.  The proof is a quantitative Morse
theory argument, like the ``pushing'' arguments in \cite{ABBDY}.  Bux
and Wortman \cite{BuxWortConnectivity} used a Morse theory argument to
prove that $Z$ is $n$-connected; we sketch their proof in the case
that $\tau$ is a generic direction.  In general, $X$ is contractible,
and $Z$ is the level set of $h$.  If $\tau$ is generic, then $h$ is
nonconstant on every edge of $X$, and we can treat it as a
combinatorial Morse function.

That is, if $u$ is a vertex of $X$, then every vertex of its link
$\Lk(u)$ corresponds to a vertex $v$ adjacent to $u$.  We define the
\emph{downward link} $\downlk{u}\subset \Lk u$ to be the subcomplex
spanned by vertices $v$ with $h(v)<h(u)$.  By results of Schulz
\cite{Schulz}, $\downlk{u}$ is $(n-2)$-connected for all $u$, so
combinatorial Morse theory implies that $Z$ is also $(n-2)$-connected.
Bux and Wortman apply a similar argument in the general case, but with
$h$ replaced by a more complicated function to deal with faces of
dimension $>0$ that are orthogonal to $\tau$.

Arguments like this, however, give poor quantitative bounds.  Given an
$(n-2)$-sphere in $Z$, one can construct a filling in the horosphere
$h^{-1}([0,\infty))$ and use Morse theory to homotope it to $Z$, but
the filling may grow exponentially large in the process.  The pushing
methods in \cite{ABBDY} avoid this sort of exponential growth by
constructing maps from $\downlk{u}$ to $Z$, and we will apply similar
methods here.  

Let $\fa$ be a chamber of $X_\infty$ which contains $\tau$ in its
closure and let
$$X_\infty^0(\fa):=\{\fb \mid \fb \text{ is a chamber of $X_\infty$ opposite
  to $\fa$}\}.$$ Abramenko \cite{AbramTwin} showed that if $Y$ is a
sufficiently thick classical spherical building, then $Y^0(\fa)$ is
$(\rk Y-2)$-connected for any chamber $\fa$ of $Y$.  We will show that
if $X$ is a thick euclidean building of rank $n$, then $X_\infty^0(\fa)$ is
$(n-2)$-connected.

Roughly, we show $(n-2)$-connectivity by showing that ``most'' pairs
of chambers $\fb,\fc\subset \xop$ are opposite to one another and that
if $E_{\fb,\fc}$ is the apartment they span, then $\pinfty
E_{\fb,\fc}\subset \xop$.  Then, for each sphere $\alpha:S^{k}\to
\xop$ with $k<n-2$, we choose a $\fc$ such that for any $\fb$ in the support of $\alpha(S^{n-2})$, $\fb$ is opposite to $\fc$ and $E_{\fb,\fc}\subset
\xop$.  We can then
contract $\alpha$ to a point in $\fc$ along geodesics.  Since
$X_\infty^0(\fa)$ is $(n-2)$-connected, there is no obstruction to
constructing a map
$$\Omega_\infty:\Delta_Z^{(n-1)}\to X_\infty^0(\fa).$$

Next, we construct a map to $Z$.  Given a point $x\in X$ and a
direction $\sigma\in X_\infty$, we let $r$ be the ray emanating from
$x$ in the direction of $\sigma$.  If $h(x)>0$ and $\sigma\in \xop$,
this ray will eventually intersect $Z$.  This provides a map
$X_\infty^0(\fa)\to Z$, but this map is not Lipschitz -- a ray may
travel a long distance before intersecting $Z$.  To fix this, we
define the downward link at infinity $\dlk(x)$ at $x$.  This is 
a subset $\dlk(x) \subset X_\infty$ of directions that point
``downward'' from $x$ (i.e., away from $\fa$).  Rays in these directions
intersect $Z$ after traveling distance $\lesssim h(x)$, so we can define
a map $i_x:\dlk(x)\to Z$ with Lipschitz constant $\lesssim h(x)$ which
sends each direction to the corresponding intersection point.

The sets $\dlk(x)$ get bigger as $x\to \fa$, and any finite subset
of $\xop$ is contained in some $\dlk(x)$.  This lets us convert
$\Omega_\infty$ into a map to $Z$; for each simplex $\Delta$, we
choose some $x_\Delta$ so that $\Omega_\infty(\Delta)\subset
\dlk(x_\Delta)$ and define (after some patching around the edges)
$$\Omega|_{\Delta}=i_{x_\Delta}\circ \Omega_\infty.$$

Finally, we show that restrictions of $\Omega$ to simplices satisfy
Lipschitz bounds.  To do this, we need some control over the Lipschitz
constants of the $i_{x_\Delta}$'s.  The Lipschitz constant of
$i_{x_\Delta}$ is on the order of $h(x_\Delta)$, so we try to bound
the $h(x_\Delta)$'s by controlling which chambers of $\xop$ we use
in fillings of spheres.  This proves the theorem.

The rest of this section is devoted to filling in the details of this
sketch.  First, in Sections~\ref{sec:buildingPrelims} and \ref{sec:folded}, we
describe our notation and define some maps and subsets that we will use
in the rest of the proof.  In Section~\ref{sec:xopApartments}, we
construct $\dlk(x)$ and show that there are
many apartments in $\dlk(x)$ .  In
Section~\ref{sec:constructingOmegaInf}, we use this fact to show that
$\xop$ is $(n-2)$-connected and to construct $\Omega_\infty$ and the
$x_\Delta$'s.  In Section~\ref{sec:constructingOmega}, we use these to
construct $\Omega$.  

\subsection{Preliminaries}\label{sec:buildingPrelims}

In this section, we fix some notation for dealing with buildings,
define some maps and subsets that will be important in the rest of the
section, and prove some of their properties.  Our primary reference is
\cite{AbramBro}.

As stated in the introduction to this section, we let $X$ be an
irreducible thick euclidean building of rank $n$, equipped with its
complete apartment system and let $X_\infty$ be its Bruhat-Tits
building.  If $E$ is an apartment of $X$, we can identify it with the
Coxeter complex of a Euclidean reflection group $W$, and we can
identify the corresponding apartment $\pinfty E\subset X_\infty$ with
the Coxeter complex of $\bar{W}$, the reflection group corresponding
to the linear parts of the elements of $W$.  

Recall that $X_\infty$ can be defined as the set of classes of parallel
unit-speed geodesic rays in $X$, where $r,r':[0,\infty)\to X$ are
parallel if $d(r(t),r'(t))$ is bounded as $t\to \infty$.  For any
$x\in X$ and any $\sigma\in X_\infty$, there is a unique ray based at
$x$ and parallel to $\sigma$ \cite[Lem.~11.72]{AbramBro}.  Given a
subset $Y\subset X$, we define $\pinfty Y$ to be its boundary at
infinity; for the subsets we will consider in this paper, $\pinfty Y$
consists of the set of parallelism classes of geodesic rays in $Y$.
If $\fd$ is a chamber of $\pinfty E$, we say that $E$ is
\emph{asymptotic to} $\fd$.

If $x\in E$, there is a conical cell $x+\fd$ based at $x$ for every
chamber $\fd$ of $\pinfty E$; we call these cells \emph{sectors}.
Note that $x+\fd$ doesn't depend on our choice of $E$; this
construction gives the same result for any apartment $E'$ such that
$\fd\subset \pinfty E'$ and $x\in E'$.

The codimension-1 cells of $E$ are called \emph{panels}.  Each panel
is contained in a codimension-1 subspace of $E$ which we call a
\emph{wall}.  Each wall divides $E$ into a pair of closed
\emph{half-apartments}.  We say that $E'$ is a \emph{ramification} of
$E$ if either $E=E'$ or $E\cap E'$ is a half-apartment.  Since $X$ is
thick, each wall is the boundary of at least three half-apartments.
We say that two chambers are \emph{adjacent} if they have disjoint
interiors and share a panel.  A sequence of chambers $C_1,\dots, C_k$
such that $C_i$ and $C_{i+1}$ are adjacent is called a \emph{gallery}
of combinatorial length $k$.  The minimal combinatorial length of a
gallery connecting two chambers is called the \emph{combinatorial
  distance} between them, and a gallery realizing this length is
called a \emph{minimal gallery}.  We denote the combinatorial distance
between $C$ and $C'$ by $\dcomb(C,C')$.

There is also a CAT(0) metric on $X$ which gives each apartment the
metric of $\R^n$.  We denote this metric by $d:X\times X\to \R$.
Likewise, there is a CAT(1) metric (the angular metric) on $X_\infty$, which we also denote
by $d$.

If $\fc,\fd\subset X_\infty$ are chambers and
$\dcomb(\fc,\fd)=\diam_{\text{comb}}(X_\infty)$, we say that $\fc$ and
$\fd$ are opposite.  Any pair of opposite chambers of $X_\infty$
determines a unique apartment of $X_\infty$
\cite[Thm.~4.70]{AbramBro}.  Indeed, if $\fc,\fd\subset X_\infty$ are
opposite chambers, then there is a unique apartment of $X$ which is
asymptotic to $\fc$ and $\fd$ \cite[Thm.~11.63]{AbramBro}.  

\subsection{Folded apartments} \label{sec:folded} In order to prove
Theorem~\ref{thm:horosphereLipConn}, we will need to understand how
apartments in $X$ are positioned relative to $\fa$.  In this section,
we describe some notions that will be useful to understand the
arrangement of apartments in $X$.

Recall that if
$E$ is an apartment of $X$ and $C\subset E$ is a chamber, there is a
retraction $\rho_{E,C}:X\to E$ such that if $C=C_1,\dots, C_k$ is a
minimal gallery in $X$, then $C=\rho(C_1),\dots, \rho(C_k)$ is a
minimal gallery in $E$.  We will use a related retraction which is
based at a chamber of $X_\infty$ rather than a chamber of $X$.

Following Abramenko and Brown \cite[11.7]{AbramBro}, if $E$ is an
apartment of $X$ and $\fc$ is a chamber of $\partial_\infty E$, we
define $\rho_{E,\fc}:X\to E$ to be the map such that if $E'$ is an
apartment of $X$ which is asymptotic to $\fc$, then
$\rho_{E,\fc}|_{E'}$ is the isomorphism $\phi_{E'}:E'\to E$ which
fixes $E\cap E'$ pointwise.  (In the case that $X$ is a tree, this is
the map obtained by ``dangling'' the tree from a point at infinity.)

Fix some apartment $F$ which is asymptotic to $\fa$ and let
$\rho=\rho_{F,\fa}$.  Note that changing the choice of $F$ changes
$\rho$ by an isomorphism; if $F'$ is asymptotic to $\fa$ and
$\phi_{F}:F\to F'$ is the isomorphism fixing $F\cap F'$ pointwise,
then $\rho_{F',\fa}=\phi_F\circ \rho_{F,\fa}$.  Furthermore, $\rho$
preserves Busemann functions centered at points in $\fa$.  In
particular, $h\circ \rho=h$.  

If $E$ is an apartment of $X$, then $\rho$ maps $E$ to $F$ by a
``folding'' process.  If $X$ is a tree, for instance, then either
$\rho|_E$ is an isomorphism $E\to F$ or it folds $E$ once.  In higher
rank buildings, $\rho|_E$ can be more complicated.  The following lemmas will help us describe
these maps.

For any chamber $C$ of $X$ and any chamber $\fc$ of $X_\infty$, we
define the direction $D_C(\fc)$ of $\rho(\fc)$ at $\rho(C)$ as follows.
Let $\overrightarrow{xy}$ be a directed line segment in $C$ in the
direction of an interior point of $\fc$.  Then
$\rho(\overrightarrow{xy})$ is a directed line segment in $F$ pointing
toward the interior of some chamber of $\partial_\infty F$.  We let
$D_C(\fc)$ be that chamber.  

\begin{lemma}\label{lem:DPreservesType}
  Let $C$ be a chamber of an apartment $E$.  Then $D_C|_{\pinfty
    E}:\pinfty E\to \pinfty F$ is a type-preserving isomorphism.
\end{lemma}
\begin{proof}
  If $E'$ is an apartment containing $C$ and asymptotic to $\fa$ and
  $\fc'\subset \pinfty E'$, we have $D_C(\fc')=\rho_\infty(\fc')$.
  If $\phi:E\to E'$ is the isomorphism fixing $E\cap E'$
  pointwise, then $D_C(\fc)=D_C(\phi_\infty(\fc))$ for any $\fc\subset \pinfty E$, so 
  $$D_C|_{\pinfty E}=\rho_\infty|_{\pinfty E'} \circ\phi_\infty.$$
  By Proposition~11.87 of \cite{AbramBro}, $\phi_\infty$ is a
  type-preserving isomorphism.  Likewise, since $\rho|_{E'}$ is the
  isomorphism fixing $E'\cap F$ pointwise, it induces a
  type-preserving isomorphism on $\pinfty E'$.  
\end{proof}

\begin{figure}
  \def\svgwidth{\textwidth}
  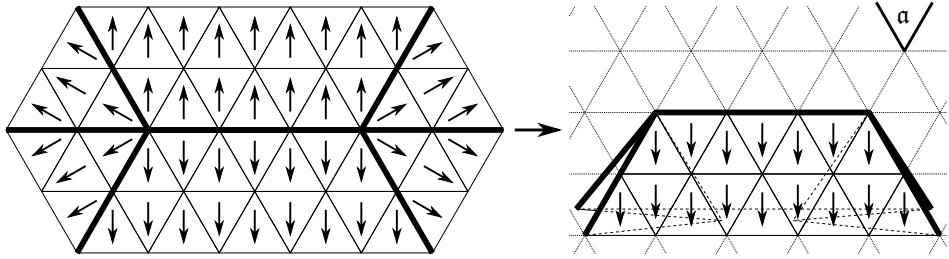
  \caption{\label{fig:folded}A subset of an apartment and its image
    under $\rho$.  (The three-dimensional effect is for clarity -- the
    map sends triangles to triangles.)  Each triangle is
    $\fa$-characteristic for the chamber of $X_\infty$ in
    the direction of its arrow.}
\end{figure}

If $C$ is a chamber of $X$, $x\in C$, and $\fc\subset X_\infty$, then
there is some subsector $x'+\fc$ of $x+\fc$ such that some apartment
of $X$ contains $x'+\fc$ and is asymptotic to $\fa$.  The proof of
Theorem 11.63 (2) in \cite{AbramBro} contains the following lemma,
which gives us a criterion for when we can take $x'=x$.
\begin{lemma}\label{lem:characteristic}
  Suppose that $E$ is an apartment of $X$ and $\fc$ is a chamber in
  $\partial_\infty E$.  If $C$ is a chamber of $E$ such that 
  $$\dcomb(\fa, D_C(\fc))=\max_{B\subset E} \dcomb (\fa, D_B(\fc))$$
  and $x\in C$, then there is an apartment of $X$ containing
  $x+\fc$ and asymptotic to $\fa$.  

  In particular, if $\fa$ and $D_C(\fc)$ are opposite, then $\fa$
  and $\fc$ are opposite.
\end{lemma}

If $C$ is a chamber of $X$ and $\fc$ is a chamber of $\partial_\infty
X$ such that $\fa$ is opposite to $D_C(\fc)$, we call $C$ an
\emph{$\fa$-characteristic chamber} for $\fc$.  
\begin{lemma}
  The following are equivalent:
  \begin{itemize}
  \item $C$ is an $\fa$-characteristic chamber for $\fc$.  
  \item $\fa$ and $\fc$ are opposite and the unique
    apartment asymptotic to $\fa$ and $\fc$ contains $C$.  
  \item $\fa$ and $\fc$ \emph{point in opposite directions at $C$}.  That is,
    whenever $x$ is in the interior of $C$, the rays from $x$ toward
    the barycenters of $\fa$ and $\fc$ point in opposite directions.
  \end{itemize}
\end{lemma}
\begin{proof}
  (1) implies (2) by Lemma~\ref{lem:characteristic}.  If (2) holds and
  $E$ is the unique apartment asymptotic to $\fa$ and $\fc$, then the
  rays toward the barycenters of $\fa$ and $\fc$ from any point in $E$
  are rays in $E$ pointing in opposite directions, so (3) holds.
  Finally, if (3) holds, then $D_C(\fa)$ and $D_C(\fc)$ are opposite chambers of $\pinfty F$.  Since $D_C(\fa)=\fa$, this implies (1).
\end{proof}
We can replace $\fa$ in the above constructions with any chamber
$\fd\subset X_\infty$, so more generally, we may say that $C$ is an
$\fd$-characteristic chamber for $\fc$ if $\fd$ and $\fc$ are opposite
and the unique apartment asymptotic to $\fd$ and $\fc$ contains $C$.
Then $C$ is an $\fd$-characteristic chamber for $\fc$ if and only if
$\fa$ and $\fc$ point in opposite directions at $C$.  

Similarly, we say that $\fc$ and $\fc'$ \emph{point in the same
  direction at $C$} if, whenever $x$ is in the interior of $C$, the
rays from $x$ toward the barycenters of $\fc$ and $\fc'$ have the same
tangent vector at $x$.  It follows that
\begin{lemma}\label{lem:sameDirection}
  If $\fc$ and $\fc'$ point in the same direction at $C$ and $C$ is
  $\fd$-characteristic for $\fc$, then it is also $\fd$-characteristic
  for $\fc'$.
\end{lemma}

We can apply this lemma to ramifications: if $C\subset E$ is
$\fa$-characteristic for $\fc\subset \pinfty E$ and $E'$ is any
apartment of $X$ that contains $C$, let $\phi:E\to E'$ be the
isomorphism fixing $E\cap E'$ pointwise and let
$\fc'=\phi_\infty(\fc)$.  Then $\fc$ and $\fc'$ point in the same
direction at $C$, so $\fc'$ is opposite to $\fa$.

Figure~\ref{fig:folded} gives an example of the possible behavior of
$\rho$ on an apartment; in the figure, $\rho$ ``folds'' $E$ along the
thick lines.  Each of the arrows is sent to an arrow pointing in the
direction opposite $\fa$, so each chamber of $E$ is
$\fa$-characteristic for the chamber of $E$ that its arrow points
toward.  Since there are arrows pointing toward every chamber of
$\pinfty E$, we have $\pinfty E\subset \xop$.  Any apartment $E'$ that
contains the pictured portion of $E$ also satisfies $\pinfty E'\subset
\xop$.  In fact, if $E'$ is such an apartment, then $\rho$ ``folds''
$E'$ in the same way as $E$ (i.e., if $\phi:E\to E'$ is the
isomorphism fixing $E\cap E'$ pointwise, then $\rho|_E=\rho|_{E'}\circ
\phi$).

As the figure suggests, every apartment can be decomposed into
$\fa$-characteristic chambers:
\begin{lemma}[{see \cite[Lem.\ 3.1.1]{DrutuFilling}}]\label{lem:drutuEuclidean}
  If $E$ is an apartment of $X$ and $\fc_1,\dots, \fc_d
  \in \partial_\infty E$ are the chambers of $\partial_\infty E$ which
  are opposite to $\fa$, then $E$ is a union of subcomplexes
  $Y_1,\dots, Y_d$ such that the chambers of $Y_i$ are the chambers of
  $E$ that are $\fa$-characteristic for $\fc_i$.  The $Y_i$'s are
  convex in the sense that if $C,C'\subset Y_i$, then any minimal
  gallery from $C$ to $C'$ is contained in $Y_i$, and the restriction
  of $\rho$ to any of the $Y_i$'s is an isomorphism.
\end{lemma}
\begin{proof}
  For each $i$, let $E_i$ be the apartment asymptotic to $\fa$ and
  $\fc_i$.  Then $Y_i=E\cap E_i$ is a convex subcomplex of $E$
  consisting of the union of the chambers of $E$ that are
  $\fa$-characteristic for $\fc_i$.  If $C$ is a chamber of $E$, let
  $\overrightarrow{xy}$ be a line segment in $\rho(C)$ in a direction
  opposite to $\fa$.  We can pull it back under $\rho$ to a line
  segment in $C$ which points in the direction of a chamber
  $\fc_i\subset \partial_\infty E$.  Then $C$ is an
  $\fa$-characteristic chamber for $\fc_i$ and $C\subset Y_i$.
\end{proof}

Even when $C$ is not $\fa$-characteristic for $\fc$, the direction
$D_C(\fc)$ still tells us about $\rho|_{x+\fc}$ for $x\in C$.  The
following lemma strengthens Lemma~\ref{lem:characteristic}.
\begin{lemma}\label{lem:characteristic2}
  Suppose that $\fc$ is a chamber in $\partial_\infty X$, that $C$ is
  a chamber of $X$, and $x_0\in C$.  Let $C'$ be a
  chamber which intersects the sector $x_0+\fc$.  Then either
  $D_{C}(\fc)=D_{C'}(\fc)$ or $\dcomb(\fa,D_{C}(\fc))<\dcomb(\fa,D_{C'}(\fc))$.
\end{lemma}
\begin{proof}
  We proceed similarly to \cite[11.63(2)]{AbramBro}.  
  
  Let $x\in C'$ be a point in $C'\cap x_0+\fc$.  We may choose $x$ so
  that the geodesic segment $\overrightarrow{x_0x}$ never crosses two
  walls simultaneously.  Then $\overrightarrow{x_0x}$ passes through
  chambers $C=C_0,\dots, C_l=C'$ which all meet $x_0+\fc$ and which
  form a minimal gallery in $X$.  For each $i$, let $x_i$ be a point on
  $\overrightarrow{x_0x}$ which lies on the interior of $C_i$.  

  We proceed inductively.  Suppose that the lemma is true for
  $C'=C_0,\dots, C_i$ and consider $C'=C_{i+1}$.  

  Let $E$ be an apartment containing
  $C_i$ and asymptotic to $\fa$.  Let $A$ be the common
  panel between $C_i$ and $C_{i+1}$ and let $H$ be the wall of $E$
  containing $A$.  Let $E^+\subset E$ be the half-apartment bounded by $H$
  which is asymptotic to $\fa$ and let $E^-\subset E$ be the opposite half-apartment.

  We consider two cases: $C_i \subset E^+$ and $C_i\subset E^-$.

  If $C_i\subset E^+$, let $E'$ be a ramification of $E$ (possibly $E$
  itself) which contains $E^+$ and $C_{i+1}$.  This is an apartment
  asymptotic to $\fa$, so by the definition of $\rho$, the restriction
  $\rho|_{E'}$ is an isomorphism fixing $E'\cap F$ pointwise.  This map 
  sends the line segment $\overrightarrow{x_ix_{i+1}}$ to the line
  segment $\overrightarrow{\rho(x_i)\rho(x_{i+1})}$.  Since
  $\overrightarrow{x_ix_{i+1}}$ is a line segment in the direction of
  an interior point of $\fc$, this implies that
  \begin{equation}
    D_{C_i}(\fc)=D_{C_{i+1}}(\fc) \label{eq:equalDirs}
  \end{equation}

  If $C_i\subset E^-$, then we have two possibilities: either
  $C_{i+1}\subset E$ or $C_{i+1}\not\subset E$.  If $C_{i+1}\subset
  E$, then the argument above, applied to $E$, shows that
  $D_{C_i}(\fc)=D_{C_{i+1}}(\fc)$.  Otherwise, let $E'$ be a
  ramification of $E$ which contains $E^-$ and
  $C_{i+1}$ and let $D=E'\smallsetminus E^-$.  Then $D\cup E^+$ is an
  apartment asymptotic to $\fa$, so $\rho|_{D}$ is an
  isomorphism.  Likewise, $\rho|_{E^-}$ is an isomorphism.  In fact, the
  restriction of $\rho$ to $E'=E^-\cup D$ is a map $E'\to F$ which ``folds''
  $E'$ along $H$, sending both $E^-$ and $D$ to $\rho(E^-)$.

  If $s:F\to F$ is the reflection fixing $\rho(H)$,
  $$D_{C_{i+1}}(\fc)=s_\infty(D_{C_{i}}(\fc)).$$  But
  $\overrightarrow{x_0x}$ passes from $E^-$ to $E^+$, so $\fc\subset
  \pinfty E^+$ and $D_{C_i}(\fc)$ is on the same side of
  $\partial_\infty \rho(H)$ as $\fa$.  Therefore,
  \begin{equation}
    \dcomb(\fa,D_{C_i}(\fc))<\dcomb(\fa,D_{C_{i+1}}(\fc)). \label{eq:dirIncrease}
  \end{equation}

  Either \eqref{eq:equalDirs} or \eqref{eq:dirIncrease} holds for each $i$.  The lemma follows by induction.  
\end{proof}

We will also define some families of subsets of $X$ and $X_\infty$.
Our argument is essentially a quantitative version of Morse theory, so
for each point $x\in X$ with $h(x)\ge 0$, we will define a set
$\dlk(x)$ of downward directions, the \emph{downward link at infinity}
and a map from that set to $Z$.  By showing that the set of downward
directions is highly connected, we will show that $Z$ is highly
connected.

For any $x\in X$, let $S(x)$ be the union of the apartments $E$ such
that $x\in E$ and $\fa\subset \partial_\infty E$.  Let
$$\dlk(x)=\pinfty S(x)\cap \xop.$$

The following properties of $\dlk(x)$ will be helpful:
\begin{lemma}\label{lem:sProps}\ \\
  \begin{enumerate}
  \item If $C$ is a chamber of $X$ and $x$ is in the interior of $C$,
    then $\fc$ is a chamber of $\dlk(x)$ if and only if $C$ is
    $\fa$-characteristic for $\fc$.
  \item If $C$ is a chamber of $X$, $x$ is in the interior of $C$, and 
    $\fc,\fc'\subset \dlk(x)$, then $\fc$ and $\fc'$ point in the same direction at $C$.
  \item If $x'\in x+\fa$, then $\dlk(x)\subset
    \dlk(x')$.
  \item If $Q\subset X$ is a bounded subset, then there is an $x\in X$
    such that $d(Q,x)\lesssim \diam Q$ and $x\in q+\fa$ for any $q\in Q$.
  \item If $r:[0,\infty)$ is a unit-speed ray emanating from $x$ in
    the direction of a point $\sigma\in \dlk(x)$, then
    $$h(r(t))=h(x)+t \cos d(\tau,\sigma).$$
    Furthermore, there is an $\epsilon>0$ depending on $X$ and $p(\tau)$ such that
    $-\cos d(\tau,\sigma)>\epsilon$.
  \end{enumerate}
\end{lemma}
\begin{proof}
  The first property follows from the definition of $\dlk(x)$ and the
  fact that $C$ is an $\fa$-characteristic chamber for $\fc$ if and
  only if $\fa$ and $\fc$ are opposite and the unique apartment
  asymptotic to $\fa$ and $\fc$ contains $C$.

  If $x$ is in the interior of $C$ and $\fc,\fc'\subset \dlk(x)$, then
  $C$ is $\fa$-characteristic for $\fc$ and $\fc'$.  Consequently,
  $D_C(\fc)$ and $D_C(\fc')$ are both the chamber of $\pinfty F$
  opposite to $\fa$, so $\fc$ and $\fc'$ point in the same direction
  at $C$.

  For the third property, we show that $S(x)\subset S(x')$.  If $y\in
  S(x)$, then there is an apartment containing $x$ and $y$ and
  asymptotic to $\fa$.  Since $x'\in x+\fa$, $x'$ lies in this
  apartment as well.  It follows that $\dlk(x)\subset \dlk(x')$.

  To prove the fourth property, for all $q\in Q$, let
  $r_q:[0,\infty)\to X$ be a ray emanating from $q$ in the direction
  of the barycenter of $\fa$.  Let $E$ be an apartment asymptotic to
  $\fa$ that intersects $Q$ nontrivially.  Then $d(q, E)\le \diam Q$
  for any $q\in Q$, so by Lemma 4.6.3 of \cite{KleinerLeeb}, there is
  a $c$ such that if $t\ge c \diam Q$, then $r_q(t)\in E$.  In
  particular, $V=\bigcap_q r_q(t)+\fa$ is a sector in $E$ that
  satisfies $V\subset q+\fa$ for all $q$ and $d(V,Q)\lesssim\diam Q$.
  Choose $x\in V$.

  Finally, if $r$ is a ray in the direction of $\sigma$, let $E$ be an
  apartment which contains $x$ and is asymptotic to $\fa$ and to
  $\sigma$.  Then $r$ is a geodesic ray in $E$, which makes an angle
  of $d(\tau,\sigma)$ with the ray emanating from $x$ in the direction
  of $\tau$.  The formula for $h(r(t))$ follows by trigonometry.  

  To bound $d(\tau,\sigma)$, consider 
  $$m=\max_{\theta\in \fa} d(\tau,\theta).$$
  If $\bar{\sigma}$ is the direction opposite to $\sigma$ in $\pinfty
  E$, then by the definition of $\dlk(x)$, we have $\bar{\sigma}\in
  \fa$, so $d(\tau,\sigma)=\pi-d(\tau,\bar{\sigma})\ge \pi-m$.  We
  claim that $m<\pi/2$.

  By Lemma~4.1 of \cite{BuxWortConnectivity}, the diameter of $\fa$ is
  at most $\pi/2$, and if the diameter is equal to $\pi/2$, then $\fa$
  is a nontrivial spherical join and $X$ is a nontrivial product of
  buildings.  Furthermore, if $\theta\in \fa$ is such that
  $d(\tau,\theta)=\pi/2$, then we can write $X=X_1\times X_2$ such
  that $\tau\in (X_1)_\infty, \theta\in (X_2)_\infty$.  This
  contradicts the hypothesis that $\tau$ is not parallel to a factor
  of $X$, so $m<\pi/2$ and $-\cos d(\tau,\sigma)\ge -\cos m>0$.
\end{proof}

\subsection{Apartments in $\xop$}\label{sec:xopApartments}

In this section, we use the tools of the previous section to construct
apartments in $\xop$; in the next section, we will use these
apartments to contract spheres in $\xop$.  First, we show that every
chamber in $\xop$ is part of some apartment in $\xop$:
\begin{lemma}\label{lem:superOppSpecial}
  Suppose that $\fc$ is a chamber of $X_\infty$ opposite to $\fa$ and
  suppose that $C$ is an $\fa$-characteristic chamber for $\fc$.
  There is an apartment $E$ containing $C$ such that $E$ is asymptotic
  to $\fc$ and every chamber of $\partial_\infty E$ is opposite to
  $\fa$.

  Furthermore, there is a $c>0$ depending only on $X$ and an
  $\fa$-characteristic chamber $C_\fb\subset E$ for each chamber
  $\fb\subset \partial_\infty E$ such that $C_\fc=C$ and
  $$\diam \bigcup_{\fb\subset \partial_\infty E} C_\fb\le c.$$
\end{lemma}

We will prove this lemma by starting with an apartment $E\subset X$,
then producing a series of ramifications of $E$ so that more and more
chambers of $\pinfty E$ are opposite to $\fa$.  Since $X$ is thick, if
$\fc$ is a chamber of $\pinfty E$ which is not opposite to $\fa$, then
there is some ramification $E'$ of $E$ that replaces $\fc$ with a
chamber that is farther (in $X_\infty$) from $\fa$.  This might
replace a chamber of $\pinfty E$ which is already opposite to $\fa$
with a chamber which is not, but we avoid this by ensuring that $E'$
contains the same $\fa$-characteristic chambers as $E$.

The following lemma produces these ramifications:
\begin{lemma}\label{lem:inductiveOpp}
  Let $E$ be an apartment of $X$ and let $\fc=\fc_1,\dots, \fc_k$ be
  chambers of $\partial_\infty E$ which are opposite to $\fa$.  Let
  $C_i\subset E$ be a $\fa$-characteristic chamber for $\fc_i$ for
  each $i$. Let $\fb$ be a chamber of $\partial_\infty
  E$, distinct from the $\fc_i$'s, which is adjacent to $\fc$.  There
  is a ramification $E_0$ of $E$ such that if $\phi:E\to E_0$ is the
  isomorphism fixing $E\cap E_0$ pointwise, then
  \begin{itemize}
  \item $C_i\subset E\cap E_0$ for all $i$ (and thus
    $\phi_\infty(\fc_i)$ is opposite to $\fa$),
  \item $\phi_\infty(\fb)$ is opposite to $\fa$, and
  \item there is an $\fa$-characteristic chamber $B_0\subset E_0$ for
    $\phi_\infty(\fb)$ such that $d(B_0,\bigcup C_i)\lesssim \diam \bigcup C_i$.
  \end{itemize}
\end{lemma}
\begin{proof}
  Let $C=C_1$ and let $x_0\in C$.  Let $H$ be a wall in $E$ such that
  $\partial_\infty H$ separates $\fb$ and $\fc$.  Let $M,M'\subset E$
  be the half-apartments of $E$ bounded by $H$.  By translating $H$ and
  possibly switching $M$ and $M'$, we may arrange that 
  \begin{itemize}
  \item $\fc\in \pinfty M$ and $\fb\in \pinfty M',$
  \item $C_i\subset M$ for all $i$, and
  \item $d(H,C)\lesssim \diam(\bigcup C_i).$
  \end{itemize}
  We claim that there is a ramification $E_0$ of $E$ which contains $M$
  and satisfies the conditions of the lemma.  

  By our choice of $H$, the intersection $x_0+\fb\cap M'$ is a sector
  of $E$, and we can choose $B\subset x_0+\fb\cap M'$ to be a chamber
  which borders $H$ and satisfies $d(x_0,B)\lesssim \diam(\bigcup
  C_i).$ Let $A$ be the panel of $H$ bordering $B$, let $D\subset M$
  be the chamber of $E$ adjacent to $B$ along $A$, and let $B'$ be a
  chamber adjacent to $A$ and distinct from $B$ and $D$.  Let $E'$ be
  a ramification of $E$ that contains $B'$ and let $\phi:E\to E'$ be
  the isomorphism fixing $E\cap E'$.  We claim that either the lemma
  is satisfied for $E_0=E$ and $B_0=B$ or it is satisfied for $E_0=E'$
  and $B_0=B'$.

  Since $\fa$ is opposite to $D_{C}(\fc)$ and $D_C(\fb)$ is adjacent to $D_C(\fc)$,
  $$\dcomb(\fa,D_{C}(\fb))=\dcomb(\fa,D_{C}(\fc))-1.$$
  Lemma~\ref{lem:characteristic2} implies that either $D_{B}(\fb)$ is
  opposite to $\fa$ or $D_{B}(\fb)=D_{C}(\fb)$.  By
  Lemma~\ref{lem:DPreservesType}, $D_{B}$ and $D_C$ are
  type-preserving isomorphisms from $\pinfty E$ to $\pinfty F$, so if
  $D_{B}(\fb)=D_{C}(\fb)$, then $D_{B}=D_{C}$, and $B$ is
  $\fa$-characteristic for $\fc$.  So $B$ is $\fa$-characteristic for
  either $\fb$ or $\fc$.  In the first case, the lemma is satisfied
  for $E_0=E$ and $B_0=B$.

  Likewise, if $\fb'=\phi_\infty(\fb)$, then $D_C(\fb')=D_C(\fb)$ is
  adjacent to $D_C(\fc)$ and $B'\subset x_0+\fb'$, so $B'$ is
  $\fa$-characteristic for either $\fb'$ or $\fc$.  In the first case,
  the lemma is satisfied for $E_0=E'$ and $B_0=B'$.  

  Suppose by way of contradiction that $B$ and $B'$ are both
  $\fa$-characteristic for $\fc$.  The union of the set of chambers of
  $X$ that are $\fa$-characteristic for $\fc$ is the unique apartment
  $E_{\fa,\fc}$ asymptotic to $\fa$ and $\fc$, so in particular, it is
  convex.  It contains $B$ and $C$, so it contains $D$ as well.  But
  then $B$, $B'$, and $D$ are distinct chambers of $E_{\fa,\fc}$ which
  are all adjacent to the same panel.  This is impossible.
\end{proof}

\begin{proof}[Proof of Lemma~\ref{lem:superOppSpecial}]
  Let $E_{\fa,\fc}\subset X$ be the apartment spanned by $\fa$ and
  $\fc$, so that $C\subset E_{\fa,\fc}$.  By applying
  Lemma~\ref{lem:inductiveOpp} to $E_{\fa,\fc}$ repeatedly, we can
  construct an apartment $E$ such that for any chamber
  $\fb\in \partial_\infty E$, there is an $\fa$-characteristic chamber
  $C_{\fb}$ for $\fb$, and
  $$\diam \bigcup_{\fb\subset \partial_\infty E'}  C_{\fb}$$
  is bounded.
\end{proof}

In fact, we can find many apartments in $\xop$ simultaneously:
\begin{lemma}\label{lem:miniSuperOpp}
  Suppose that $E$ is an apartment of $X$ and suppose that for each
  chamber $\fc\subset \pinfty E$ there is a chamber $C_\fc\subset E$
  which is $\fa$-characteristic for $\fc$ and a point $x_\fc\in
  C_\fc$.  Let $\fb$ and $\bar\fb$ be two opposite chambers in
  $\pinfty E$.  Suppose that $C$ is a chamber of $X$ and $x$ is a
  point in the interior of $C$ such that $x\in x_\fb+\fb$ and
  $C_\fc\subset x+\bar\fb$ for all $\fc \subset \pinfty E$.  Then
  there is an $x'\in x+\fa$ such that 
  $$d(x,x')\lesssim \diam \bigcup_{\fc\subset \pinfty E} C_\fc$$
  and for every chamber $\fd\subset\dlk(x)$,
  \begin{itemize}
  \item $\fd$ is opposite to $\bar\fb$,
  \item if $E_{\fd,\bar\fb}$ is the apartment spanned by $\fd$ and
    $\bar\fb$, then $\partial_\infty E_{\fd,\bar\fb} \subset
    \dlk(x')$.
  \end{itemize}
\end{lemma}
\begin{proof}
  Suppose that $\fd\subset\dlk(x)$. Then $C$ is
  $\fa$-characteristic for $\fb$ and $\fd$, so $\fb$ and $\fd$ point
  in the same direction at $C$.  Since $\fb$ and $\bar\fb$ point in
  opposite directions at $C$, we conclude that $C$ is
  $\bar\fb$-characteristic for $\fd$.  Thus, $\bar\fb$ and $\fd$ are
  opposite and $C\subset E_{\fd,\bar\fb}$.

  In particular, $x+\bar\fb\subset E_{\fd,\bar\fb}$, so
  $C_\fc\subset E_{\fd,\bar\fb}$ for all $\fc\subset \pinfty E$.  Let
  $\phi:E_{\fd,\bar\fb}\to E$ be the isomorphism fixing
  $E_{\fd,\bar\fb}\cap E$ pointwise and suppose that $\fc'\subset
  \pinfty E_{\fd,\bar\fb}$.  If $\fc=\phi_\infty(\fc')$, then
  $C_{\fc}$ is an $\fa$-characteristic chamber for $\fc'$, so
  $\fc'\subset \dlk(x_\fc)$.  

  By Lemma~\ref{lem:sProps}(3) and (4), there is an
  $$x'\in \bigcap_{\fc\subset \pinfty E} x_\fc+\fa.$$
  such that $x'\in x+\fa$ and $\dlk(x_\fc)\subset \dlk(x')$ for every
  $\fc\subset \pinfty E$.
\end{proof}

Combining Lemmas~\ref{lem:miniSuperOpp} and \ref{lem:superOppSpecial}
we get:
\begin{lemma}\label{lem:superOpp}
  For any $x\in X$, there is a chamber $\fd\subset X_\infty$ opposite
  to $\fa$ and an $x'\in x+\fa$ such that
  \begin{itemize}
  \item if $\fc\subset \dlk(x)$ then $\fd$ is
    opposite to $\fc$,
  \item if $\fc\subset \dlk(x)$ and $E_{\fc,\fd}$ is the apartment spanned by $\fc$ and $\fd$, then $\partial_\infty E_{\fc,\fd} \subset \dlk(x')$, and
  \item $d(x,x')\lesssim 1$.
  \end{itemize}
\end{lemma}
\begin{proof}
  Let $\fb\subset \dlk(x)$ and let $E$ be the unique apartment
  asymptotic to $\fa$ and $\fb$.  Since $\fb\subset \dlk(x)$, we have
  $x\in E$.  We may perturb $x$ in the direction of $\fa$ to ensure
  that $x$ is in the interior of some chamber $C$ of $E$; this doesn't
  change $\dlk(x)$.  Let $r$ be a unit-speed ray emanating from $x$ in
  the direction of the barycenter of $\fa$ and let $0<\theta<\pi/2$ be
  the minimum angle between the barycenter of $\fa$ and any point on
  its boundary.  Let $c$ be the constant in
  Lemma~\ref{lem:superOppSpecial} and let $t>\frac{c}{\sin{\theta}}$,
  so that
  $$B_{E}(r(t),c) \subset x+\fa,$$ 
  where $B_{E}(r(t),c)$ is the ball in $E$ with center $r(t)$ and radius $c$.  Let $x_0=r(t)$.

  Let $C_0\subset E$ be a chamber such that $x_0\in C_0$.
  Since $C_0\subset E$, it is $\fa$-characteristic for $\fb$.  By
  Lemma~\ref{lem:superOppSpecial}, there is an apartment $E'$ and a
  collection of $\fa$-characteristic chambers $C_\fc\subset E'$ for
  $\fc\subset \partial_\infty E'$ such that $x_0+\fb\subset E'$ and
  $$\bigcup_{\fc\subset \pinfty E'} C_\fc\subset B_{E'}(x_0, c).$$
  Let $\bar\fb$ be the chamber of
  $\partial_\infty E'$ opposite to $\fb$.  We claim that $x+\bar\fb$
  contains all of the $C_\fc$'s.  

  Let $\phi: E\to E'$ be the isomorphism fixing $E\cap E'$ pointwise.
  Then $\phi$ fixes $C$ and $C_0$ and sends $\fa$ to $\bar\fb$, so
  $\phi(x+\fa)=x+\bar\fb$ and $\phi(B_E(x_0,c))=B_{E'}(x_0,c)$.
  Therefore, 
  $$\bigcup_{\fc\subset \pinfty E'} C_\fc\subset B_{E'}(x_0,t)\subset x+\bar\fb.$$
  By applying Lemma~\ref{lem:miniSuperOpp} to $E'$, we obtain an $x'$
  that satisfies the required properties and has 
  $$d(x,x')\lesssim \diam \bigcup_{\fc\subset \pinfty E'} C_\fc\lesssim 1.$$
\end{proof}

We can also use these techniques to construct $(n-1)$-spheres in $Z$
which are homotopically nontrivial in $Z$.  This generalizes results
of Bux and Wortman \cite{BuxWortFiniteness} on buildings acted on by
$S$-arithmetic groups to arbitrary euclidean buildings.
\begin{lemma}\label{lem:nonFiniteness}
  For any $r>0$, there is a map $\alpha:S^{n-1}\to Z$ such that
  $\alpha$ is homotopically nontrivial in $N_r(Z)$, where $N_r(Z)$ is
  the $r$-neighborhood of $Z$.
\end{lemma}
\begin{proof}
  Let $C$ be a chamber of $X$ such that $\min_{x\in C} h(x)>r$.  Let
  $E$ be an apartment containing $C$ and asymptotic to $\fa$.  If
  $\fc\subset \pinfty E$ is the chamber of $\pinfty E$ opposite to
  $\fa$, then $C$ is $\fa$-characteristic for $\fc$.  Using
  Lemma~\ref{lem:inductiveOpp}, we can construct an apartment $E'$
  such that $C\subset E'$ and $\pinfty E'\subset \xop$.  In
  particular, the set of points $B=\{x\in E'\mid h(x)\ge 0\}$ is convex
  and compact and contains $C$, so $Z\cap E'$ is bilipschitz
  equivalent to the $(n-1)$-sphere.  Let $\alpha:S^{n-1}\to Z\cap E'$
  be a Lipschitz homeomorphism.  We claim that $\alpha$ is
  homotopically nontrivial in $N_r(Z)$.

  Let $\beta:D^n\to E$ be a homeomorphism from $D^n\to B$ which
  extends $\alpha$.  This has degree 1 on any point in the interior of
  $C$.  By way of contradiction, suppose that $\beta':D^n\to N_r(Z)$
  is another extension of $\alpha$.  Then we can glue $\beta$ and
  $\beta'$ together to get a map $\gamma:S^n\to X$.  Since $\beta'$
  avoids $C$, this map has degree 1 on any point in the interior of
  $C$.  Since $X$ is CAT(0), however, it is contractible, so $\gamma$
  must be null-homotopic, and $\gamma$ sends the fundamental class of
  $S^n$ to an $n$-boundary in $X$.  This contradicts the fact that
  this map has degree 1 on any point in the interior of $C$, because
  $X$ is $n$-dimensional, and any $n$-boundary must be trivial.  
\end{proof}

\subsection{$(n-2)$-connectivity for $\xop$ and constructing $\Omega_\infty$}\label{sec:constructingOmegaInf}

The lemmas of the previous section will let us prove that $\xop$ is
$(n-2)$-connected and construct a Lipschitz map 
$$\Omega_\infty:\Delta_Z^{(n-1)}\to X_\infty^0(\fa)$$ 
which we will use to construct $\Omega$.

Let $\Delta_Z$ be the
infinite-dimensional simplex with vertex set $Z$.  As before, we denote the
simplex of $\Delta_Z$ with vertices $z_0,\dots,z_k$ by $\langle
z_0,\dots,z_k\rangle$.  If $\Delta$ is a simplex of
$\Delta_Z$, we let $\V(\Delta)\subset Z$ be the vertex set of
$\Delta$.

The main lemma of this section is the following:
\begin{lemma}\label{lem:constructOmegaInf}
  There is a cellular map 
  $$\Omega_\infty:\Delta_Z^{(n-1)}\to X_\infty^0(\fa),$$
  a $c>0$ depending on $X$, and a family of points $x_\Delta\in X$,
  one for each simplex $\Delta\subset \Delta_Z^{(n-1)}$, such that 
  \begin{enumerate}
  \item $\mass(\Omega_\infty(\Delta))\le c$
  \item $h(x_\Delta)\ge 0$,
  \item $\Omega_\infty(\Delta)\subset \dlk(x_\Delta),$
  \item if $\Delta'\subset \Delta$, then $x_\Delta\in
    x_{\Delta'}+\fa$,
  \item and $d(x_\Delta,\V(\Delta))\lesssim \diam \V(\Delta)+1$ (consequently, $h(x_\Delta)\lesssim \diam \V(\Delta)+1$),
  \end{enumerate}
  Furthermore, for any $z\in Z$, we have $x_{\langle z\rangle}=z$.
\end{lemma}
The first condition is essentially a bound on the filling functions of
$\xop$.  The next three conditions ensure that the map $i_{x_\Delta}$ (as defined in the proof sketch at the beginning of the section)
is defined on $\Omega_\infty(\Delta)$ and that its Lipschitz constant
is $\lesssim \diam \V(\Delta)$.  In order to construct $\Omega$ in
the next section, we will glue maps of the form $i_{x_\Delta}\circ
\Omega_\infty|_\Delta$, and we will use the last condition to perform
this gluing.

First, we prove that $\xop$ is
$(n-2)$-connected.  
\begin{lemma}\label{lem:xopConnected}
  If $k<n-1$, there is a $c>0$ such that for every $x\in X$, there is
  a $x'\in X$ such that $x'\in x+\fa$, $d(x,x')\le c$, and if
  $$\alpha:S^k\to \dlk(x)^{(k)},$$
  then there is an extension 
  $$\beta:B^{k+1}\to \dlk(x')^{(k+1)}$$
  such that $\Lip\beta\le c'\Lip\alpha+c'$.

  Consequently, $\xop$ is $(n-2)$-connected.
\end{lemma}
\begin{proof}
  Let $x'\in X$ and $\fd\subset \dlk(x')$ be opposite to every chamber
  of $\dlk(x)$ as in Lemma~\ref{lem:superOpp}.  Let $u$ be the
  barycenter of $\fd$.  There is an $\epsilon>0$ such that
  $d_{X\infty}(u,v)<\pi-\epsilon$ for any $v\in \dlk(x)^{(k)}$.  By
  our choice of $\fd$, the geodesic from $v$ to $u$ is contained in
  $\dlk(x')$.
  
  Let 
  $$\gamma:\dlk(x)^{(k)}\times [0,1]\to \dlk(x')$$ 
  be the map which sends $v\times [0,1]$ to the geodesic between $v$
  and $u$.  This is Lipschitz, with Lipschitz constant depending on
  $\epsilon$.  Define $\beta_0:S^k\times [0,1]\to \dlk(x')$ by
  $\beta_0(v,t)=\gamma(\alpha(v),t)$.  This is a null-homotopy of
  $\alpha$, and
  $$\Lip \beta_0\le (\Lip\gamma)(1+\Lip \alpha).$$
  We obtain $\beta$ by approximating $\beta_0$ in
  $\dlk(x)^{(k+1)}$; this increases the Lipschitz constant by at
  most a multiplicative factor.  

  To conclude that $\xop$ is $(n-2)$-connected, consider a map
  $\alpha:S^k\to \xop$.  This can be approximated by a simplicial map
  $\alpha':S^k\to \xop^{(k)}$.  The image of $\alpha'$ has finitely
  many simplices, and since every simplex of $\xop$ is contained in
  $\dlk(y)$ for some $y$, there is an $x\in X$ such that the image
  of $\alpha'$ is contained in $\dlk(x)^{(k)}$.
  Therefore, $\alpha'$ is null-homotopic in $\dlk(x')^{(k)}$ for
  some $x'$, and $\dlk(x')^{(k)}\subset\xop$.
\end{proof}

Next, we use this lemma to construct $\Omega_\infty$:
\begin{proof}[{Proof of Lemma~\ref{lem:constructOmegaInf}}]
  We construct $\Omega_\infty$ inductively.  First, for each $z\in Z$,
  we let $\Omega_\infty(z)$ be an arbitrary vertex of $\dlk(z)$.  Then
  choosing $x_{\langle z\rangle}=z$ satisfies the
  conditions of the lemma.

  Now suppose that $\Delta$ is a simplex of $\Delta_Z$ with $1\le \dim
  \Delta=k \le n-1$ and suppose that $\Omega_\infty$ is defined on
  $\partial \Delta$.  Then, if
  $$L=\bigcup_{\Delta'\subset \Delta}\dlk(x_{\Delta'}),$$
  then $\Omega_\infty|_{\partial \Delta}$ is a map with image in
  $L$.  By induction, we know that 
  $$d(x_{\Delta'},\V(\Delta'))\lesssim_k \diam \V(\Delta')+1,$$
  with implicit constant depending on $k$, so
  $$\diam \{x_{\Delta'}\}_{\Delta'\subset \Delta}\lesssim_k \diam \V(\Delta)+1.$$
  By Lemma~\ref{lem:sProps}(4), there is an $x_0\in X$ such that
  $d(x_0,\V(\Delta))\lesssim_k \diam \V(\Delta)+1$ and $x_0\in x_{\Delta'}+\fa$
  for any face $\Delta'$ of $\Delta$. By Lemma~\ref{lem:sProps}(3),
  $L\subset \dlk(x_0)$.

  By Lemma~\ref{lem:xopConnected}, there is an $x' \in X$ such that
  $x'\in x_0+\fa$ and an extension
  $$\beta:\Delta \to \dlk(x')^{(k+1)}$$
  of $\alpha$ such that $\Lip\beta$ and $d(x_0,x')$ are bounded by a
  constant depending on $k$.  If we define
  $\Omega_\infty|_{ \Delta}=\beta$ and $x_\Delta=x'$, then 
  $$d(x_\Delta,\V(\Delta))\lesssim_k \diam \V(\Delta)+1.$$
  Since $\Delta$ is finite-dimensional, we may drop the dependence on
  $k$, and the lemma holds.
\end{proof}

\subsection{Constructing $\Omega$}\label{sec:constructingOmega}

Finally, we construct a map $\Omega:\Delta_Z^{(n-1)}\to Z$ satisfying
the hypotheses of Lemma~\ref{lem:LipToInfSimp}.  We will use a
family of maps $i_x:\dlk(x)\to Z$ for $x\in X, h(x)\ge 0$.

For any $x\in X$ and $\sigma\in X_\infty$, there is a unit-speed ray
$r_\sigma:[0,\infty)\to X$ emanating from $x$ and traveling in the
direction of $\sigma$.  Define 
$$X_\infty^*=X_\infty \times [0,\infty)/X_\infty\times\{0\}$$
to be a space of ``vectors'' based at $x$.  
We can define an
exponential map $e_x:X_\infty^*\to X$ by letting
$$e_x(\sigma,t)=r_\sigma(t).$$
For each chamber $\fa$ of $X_\infty$, this map sends the open cone $\fa
\times [0,\infty)/\fa\times\{0\}$ to a sector corresponding to
$\fa$; we give $X_\infty^*$ a metric so that this is an isometry.  This
makes $e_x$ a distance-decreasing map.  Note also that, by the
convexity of the distance function on $X$, we have
$$d(e_x(\sigma,t), e_{x'}(\sigma,t))\le d(x,x').$$

We can use $e_x$ to construct a map from $\dlk(x)$ to
$Z$:
\begin{lemma}
  Let $x\in X$ be such that $h(x)\ge 0$.  Then there is a map $i_x:
  \dlk(x) \to Z$ given by
  $$i_x(\sigma)=e_x(\sigma, \frac{-h(x)}{\cos d(\tau,\sigma)}).$$
  This map has Lipschitz constant $\lesssim h(x)$, with implicit
  constant depending on $X$ and $p(\tau)$.
\end{lemma}
\begin{proof}
  By Lemma~\ref{lem:sProps}(5), 
  $$h(e_{x}(\sigma,t))=h(x)+t \cos d(\tau,\sigma)$$
  for any $\sigma\in \dlk(x)$ and there is an $\epsilon$ such that
  $-\cos d(\tau,\sigma)\ge \epsilon>0$.  The lemma follows.
\end{proof}

Furthermore, the map $(x,\sigma)\mapsto i_x(\sigma)$ is locally
Lipschitz:
\begin{lemma}\label{lem:iLocallyLipschitz}
  Let $x,x' \in X$ be such that $h(x),h(x')\ge 0$.  Let
  $\sigma,\sigma'\in \dlk(x)\cap \dlk(x')$.  Then
  there is a $c>0$ depending on $X$ such that
  $$d(i_x(\sigma),i_{x'}(\sigma'))\le c d(x,x')+c h(x) d(\sigma,\sigma').$$  
\end{lemma}
\begin{proof}
  By the previous lemma and the remark before it, there is a $c_0>0$ such that
  \begin{align*}
    d(i_x(\sigma),i_{x'}(\sigma'))& \le d(i_x(\sigma),i_{x}(\sigma'))+d(i_{x}(\sigma'),i_{x'}(\sigma'))\\
    & \le c_0 h(x) d(\sigma,\sigma') +d\Bigl(e_{x}\bigl(\sigma',
    \frac{-h(x)}{\cos d(\tau,\sigma')}\bigr),e_{x'}\bigl(\sigma',
    \frac{-h(x')}{\cos d(\tau,\sigma')}\bigl)\Bigr)\\
    & \le  c_0 h(x) d(\sigma,\sigma')+d(x,x')+\frac{|h(x)-h(x')|}{-\cos d(\tau,\sigma')}.
  \end{align*}
  Since $d(x,x')\lesssim h(x')$, the lemma follows.
\end{proof}

\begin{figure}
  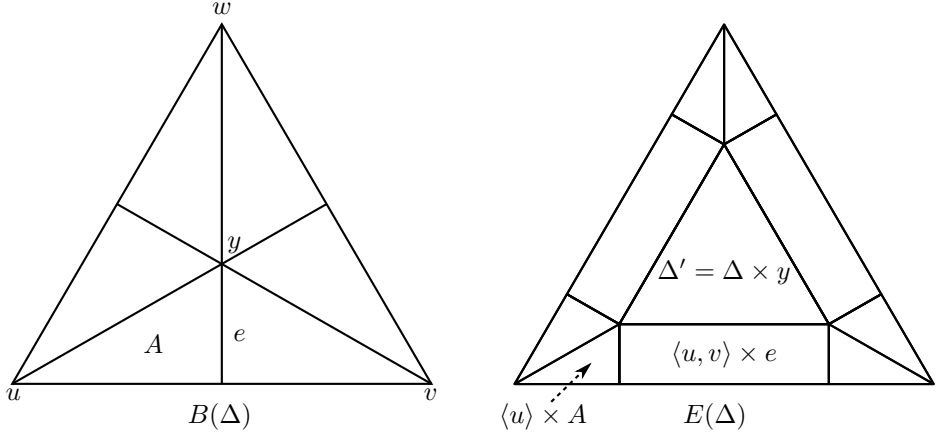
  \caption{\label{fig:triangles} Cells of the ``exploded simplex''
    $E(\Delta)$ are naturally products of cells of $\Delta$ and cells
    of $B(\Delta)$.}
\end{figure} 
We construct $\Omega$ by piecing together maps of the form
$i_{x_\Delta}(\Omega_\infty(\Delta))$, where $\Delta$ ranges over the
simplices of $\Delta_Z$.  The main problem is that if $\Delta'$ is a
face of $\Delta$, the maps $i_{x_\Delta}(\Omega_\infty(\Delta))$ and
$i_{x_{\Delta'}}(\Omega_\infty(\Delta'))$ need not agree, since
$x_\Delta\ne x_{\Delta'}$, so we need to add some ``padding'' to make
these maps agree.

Part of the construction is illustrated in Figure~\ref{fig:triangles}: for
each simplex $\Delta$ of $\Delta_Z$, we ``explode'' the barycentric
subdivision $B(\Delta)$ to get a complex $E(\Delta)$ by inserting a
copy $\Delta'$ of $\Delta$ in the middle.  Each cell in this
subdivision is of the form $\Delta_1\times \Delta_2$, where $\Delta_1$
is a face of $\Delta$ and $\Delta_2$ is a face of $B(\Delta)$.  To be
more specific, note that we can label each vertex of $B(\Delta)$ by a
face $\delta$ of $\Delta$, and the vertex labels of a simplex $\langle
\delta_0,\dots, \delta_k\rangle$ form a flag
$\delta_0\subset \dots \subset \delta_k$.  Then each cell of
$E(\Delta)$ is of the form
$$\delta\times \langle \delta_0,\dots,\delta_k\rangle,$$
for some flag $\delta_0\subset \dots\subset \delta_k$ in $\Delta$ and
some face $\delta$ of $\delta_0$.  The map $\rho_1:E(\Delta)\to
\Delta$ which projects each simplex to its first factor is a
continuous map which sends $\Delta'$ homeomorphically to
$\Delta$.  Likewise, the map $\rho_2:E(\Delta)\to B(\Delta)$ which
projects each cell to its second factor is a continuous map that
collapses $\Delta'$ to the barycenter of $\Delta$.  

We define a map $x:B(\Delta)\to X$ on the vertices of
$B(\Delta)$ by sending the point $\langle \delta\rangle$ to the point
$x_\delta$ for every face $\delta\subset \Delta$.  We define $x$ on the rest of $B(\Delta)$ by linear
interpolation.  That is, if $\delta_0\subset \dots \subset \delta_k$ is a flag of
faces of $\Delta$, then we have
$x_{\delta_i}\in x_{\delta_0}+a$ for all $i$.  Therefore, all the
$x_{\delta_i}$ lie in a common apartment, and we can define $x$ on
$\langle \delta_0, \dots, \delta_k\rangle$ by linearly interpolating 
between the $x_{\delta_i}$'s.  This map has Lipschitz constant
$\lesssim \diam \V(\Delta)$ on $\Delta$.

For any cell
$$\sigma=\delta\times \langle \delta_0,\dots,\delta_k\rangle$$
of $E(\Delta)$ and any $s\in \sigma$, let $x_s=x(\rho_2(s))$.
We have
$x_s\in x_\delta+a$ and therefore
$$\Omega_\infty(\delta)\subset \dlk(x_\delta)\subset \dlk(x_s)$$
This means that 
$$i_{x_s}(\Omega_\infty(\rho_1(s)))$$
is defined for every $s\in \sigma$, so we define
$$\Omega(s)=i_{x_s}(\Omega_\infty(\rho_1(s))).$$

Finally, we check that this definition satisfies the conditions of
Lemma~\ref{lem:LipToInfSimp}.  Since $x_{\langle z\rangle}=z$ for any
$z\in Z$, we have $\Omega(\langle z\rangle)=z$, so the first condition
is satisfied.  Let $\sigma$ be a cell of $E(\Delta)$ as above and let
$s,t\in \sigma$.  Let $x_s=x(\rho_2(s))$, $x_t=x(\rho_2(t))$.
By
Lemma~\ref{lem:iLocallyLipschitz}, we have
$$d(\Omega(s),\Omega(t))\le c d(x_s,x_t)+c h(x_s) d(\Omega_\infty(\rho_1(s)),\Omega_\infty(\rho_1(t))).$$
Since $x_\Delta\in x_{\delta_i}+\fa$ for each $i=1,\dots k$, we have
$x_\Delta\in x_s+\fa$ and thus $h(x_s)\lesssim \diam
\V(\Delta)$.  Since $\rho_1$, $\rho_2$, and $\Omega_\infty$ are
Lipschitz with constants depending only on $X$ and $\Lip(x|_\Delta)\lesssim \diam
\V(\Delta)$, each term in the inequality above is $\lesssim\diam
\V(\Delta) d(s,t)$.  Therefore, 
$$\Lip(\Omega_\infty|_\Delta)\lesssim \diam \V(\Delta)$$
for every simplex $\Delta\subset \Delta_Z^{(n-1)}$, as desired.

This proves Lemma~\ref{lem:CoarseInfSimp}.

\bibliography{higherSol}

\newcommand{\etalchar}[1]{$^{#1}$}
\def\cprime{$'$} \def\lfhook#1{\setbox0=\hbox{#1}{\ooalign{\hidewidth
  \lower1.5ex\hbox{'}\hidewidth\crcr\unhbox0}}}
\begin{thebibliography}{ABD{\etalchar{+}}12}

\bibitem[AB08]{AbramBro}
P.~Abramenko and K.~S. Brown.
\newblock {\em Buildings}, volume 248 of {\em Graduate Texts in Mathematics}.
\newblock Springer, New York, 2008.
\newblock Theory and applications.

\bibitem[ABD{\etalchar{+}}12]{ABBDY}
A.~Abrams, N.~Brady, P.~Dani, M.~Duchin, and R.~Young.
\newblock Pushing fillings in right-angled artin groups.
\newblock {\em Journal of the London Mathematical Society}, 2012.

\bibitem[Abr96]{AbramTwin}
P.~Abramenko.
\newblock {\em Twin buildings and applications to {S}-arithmetic groups},
  volume 1641 of {\em Lecture Notes in Mathematics}.
\newblock Springer-Verlag, Berlin, 1996.

\bibitem[AK00]{AmbrosioKirchheim}
L.~Ambrosio and B.~Kirchheim.
\newblock Currents in metric spaces.
\newblock {\em Acta Math.}, 185(1):1--80, 2000.

\bibitem[BEW]{BestEskWort}
M.~Bestvina, A.~Eskin, and K.~Wortman.
\newblock {Filling boundaries of coarse manifolds in semisimple and solvable
  arithmetic groups}.
\newblock arXiv:1106.0162.

\bibitem[BW07]{BuxWortFiniteness}
K.-U. Bux and K.~Wortman.
\newblock Finiteness properties of arithmetic groups over function fields.
\newblock {\em Invent. Math.}, 167(2):355--378, 2007.

\bibitem[BW11]{BuxWortConnectivity}
K.-U. Bux and K.~Wortman.
\newblock Connectivity properties of horospheres in {E}uclidean buildings and
  applications to finiteness properties of discrete groups.
\newblock {\em Invent. Math.}, 185(2):395--419, 2011.

\bibitem[dCT10]{DeCorTess}
Y.~de~Cornulier and R.~Tessera.
\newblock Metabelian groups with quadratic {D}ehn function and
  {B}aumslag-{S}olitar groups.
\newblock {\em Confluentes Math.}, 2(4):431--443, 2010.

\bibitem[Dru04]{DrutuFilling}
C.~Dru{\c{t}}u.
\newblock Filling in solvable groups and in lattices in semisimple groups.
\newblock {\em Topology}, 43(5):983--1033, 2004.

\bibitem[ECH{\etalchar{+}}92]{ECHLPT}
D.~B.~A. Epstein, J.~W. Cannon, D.~F. Holt, S.~V.~F. Levy, M.~S. Paterson, and
  W.~P. Thurston.
\newblock {\em Word processing in groups}.
\newblock Jones and Bartlett Publishers, Boston, MA, 1992.

\bibitem[Ger93]{GerstenSurv}
S.~M. Gersten.
\newblock Isoperimetric and isodiametric functions of finite presentations.
\newblock In {\em Geometric group theory, {V}ol.\ 1 ({S}ussex, 1991)}, volume
  181 of {\em London Math. Soc. Lecture Note Ser.}, pages 79--96. Cambridge
  Univ. Press, Cambridge, 1993.

\bibitem[Gro]{Groft}
C.~Groft.
\newblock Generalized {D}ehn functions {II}.
\newblock arXiv:0901.2317.

\bibitem[Gro83]{GroFRM}
M.~Gromov.
\newblock Filling {R}iemannian manifolds.
\newblock {\em J. Differential Geom.}, 18(1):1--147, 1983.

\bibitem[Gro93]{GroAI}
M.~Gromov.
\newblock Asymptotic invariants of infinite groups.
\newblock In {\em Geometric group theory, Vol.\ 2 (Sussex, 1991)}, volume 182
  of {\em London Math. Soc. Lecture Note Ser.}, pages 1--295. Cambridge Univ.
  Press, Cambridge, 1993.

\bibitem[Gro96]{GroCC}
M.~Gromov.
\newblock Carnot-{C}arath\'eodory spaces seen from within.
\newblock In {\em Sub-Riemannian geometry}, volume 144 of {\em Progr. Math.},
  pages 79--323. Birkh\"auser, Basel, 1996.

\bibitem[KL97]{KleinerLeeb}
B.~Kleiner and B.~Leeb.
\newblock Rigidity of quasi-isometries for symmetric spaces and {E}uclidean
  buildings.
\newblock {\em Inst. Hautes \'Etudes Sci. Publ. Math.}, (86):115--197 (1998),
  1997.

\bibitem[LMR00]{LMR}
A.~Lubotzky, S.~Mozes, and M.~S. Raghunathan.
\newblock The word and {R}iemannian metrics on lattices of semisimple groups.
\newblock {\em Inst. Hautes \'Etudes Sci. Publ. Math.}, (91):5--53 (2001),
  2000.

\bibitem[LP96]{LeuzPitRk2}
E.~Leuzinger and C.~Pittet.
\newblock Isoperimetric inequalities for lattices in semisimple {L}ie groups of
  rank {$2$}.
\newblock {\em Geom. Funct. Anal.}, 6(3):489--511, 1996.

\bibitem[LS05]{LangSch}
U.~Lang and T.~Schlichenmaier.
\newblock Nagata dimension, quasisymmetric embeddings, and {L}ipschitz
  extensions.
\newblock {\em Int. Math. Res. Not.}, (58):3625--3655, 2005.

\bibitem[Pit95]{PittetHMG}
C.~Pittet.
\newblock Hilbert modular groups and isoperimetric inequalities.
\newblock In {\em Combinatorial and geometric group theory ({E}dinburgh,
  1993)}, volume 204 of {\em London Math. Soc. Lecture Note Ser.}, pages
  259--268. Cambridge Univ. Press, Cambridge, 1995.

\bibitem[{Sch}]{Schulz}
B.~{Schulz}.
\newblock {Spherical subcomplexes of spherical buildings}.
\newblock arXiv:1007.2407.

\bibitem[Wen08]{WengerShort}
S.~Wenger.
\newblock A short proof of {G}romov's filling inequality.
\newblock {\em Proc. Amer. Math. Soc.}, 136(8):2937--2941, 2008.

\bibitem[Whi84]{White}
B.~White.
\newblock Mappings that minimize area in their homotopy classes.
\newblock {\em J. Differential Geom.}, 20(2):433--446, 1984.

\bibitem[You]{YoungQuad}
R.~Young.
\newblock {The Dehn function of {${\rm SL}(n;\Bbb Z)$}}.
\newblock arXiv:0912.2697.

\end{thebibliography}
\end{document}